\documentclass[11pt]{amsart}
\usepackage{bbm}
\usepackage{pdflscape}
\usepackage{pdfpages}
\usepackage{pgf}
\usepackage{tikz}
\usetikzlibrary{shapes,arrows,positioning, decorations.pathmorphing}
\usepackage[english]{babel}
\usepackage{color}
\usepackage{anysize}
\usepackage{geometry}
\marginsize{2.0cm}{2.0cm}{2.0cm}{2.0cm}
\usepackage{hyperref}
\usepackage{url}
\usepackage[normalem]{ulem}
\usepackage{hyperref}
\usepackage{mathtools}
\usepackage{graphicx}
\usepackage{amssymb}
\mathtoolsset{showonlyrefs} 
\theoremstyle{plain}
\numberwithin{equation}{section}
\newtheorem{theorem}{Theorem}[section]
\newtheorem{corollary}[theorem]{Corollary}

\newtheorem{lemma}[theorem]{Lemma}
\newtheorem{example}[theorem]{Example}
\newtheorem{remark}[theorem]{Remark}
\newtheorem{hypothesis}[theorem]{Hypothesis}
\newtheorem{definition}[theorem]{Definition}
\newcommand{\ud}{\mathrm{d}}
\newcommand{\ra}{\rightarrow}
\newcommand{\la}{\lambda}
\newcommand{\e}{\varepsilon}
\newcommand{\NN}{\mathbb{N}}
\newcommand{\RR}{\mathbb{R}}
\newcommand{\CC}{\mathbb{C}}
\newcommand{\EE}{\mathbb{E}}
\newcommand{\Wp}{\mathcal{W}_p}
\newcommand{\wW}{\mathcal{W}}
\newcommand{\bT}{\mathcal{T}}

\title[Ergodicity bounds and cutoff stability for stable OU systems]{Ergodicity bounds for stable Ornstein-Uhlenbeck systems in Wasserstein distance with applications to cutoff stability}

\author{Gerardo Barrera\\
\textnormal{ University of Helsinki, Department of Mathematics and Statistics.\\
P.O. Box 68, Pietari Kalmin katu 5, FI-00014 Helsinki, Finland.}\\
gerardo.barreravargas@helsinki.fi\\ \hfill\\}
\author{\hfill\\\hfill\\
Michael A. H\"ogele\\
Universidad de los Andes, Facultad de Ciencias, Departamento de Matem\'aticas. \\
Cra 1 \# 18A - 12, 111711 Bogot\'a, Colombia\\
ma.hoegele@uniandes.edu.co}
\date{\today}

\begin{document}
\begin{abstract}
{
This article establishes cutoff stability also known as abrupt thermalization for generic multidimensional Hurwitz stable Ornstein-Uhlenbeck systems with (possibly degenerate) L\'evy noise at fixed noise intensity. The results are based on several ergodicity quantitative lower and upper bounds some of which make use of the recently established shift linearity property of the Wasserstein-Kantorovich-Rubinstein distance by the authors. It covers such irregular systems like Jacobi chains and more general networks of coupled harmonic oscillators with a heat bath (including L\'evy excitations) at constant temperature on the outer edges and the so-called Brownian gyrator. }
\end{abstract}

\maketitle

\bigskip
\section{\textbf{Introduction}}

\textbf{
The Wasserstein-Kantorovich-Rubinstein (WKR) metric is a statistically robust and computationally flexible metric between different probability laws. Certain replica techniques allow to establish new upper and lower bounds for the thermalization for Ornstein-Uhlenbeck systems driven by Brownian motion or other L\'evy drivers. We show, that in the case of the 1d linear oscillator with Brownian forcing  and the Brownian gyrator, lengthy explicit calculations allow to establish the property of cutoff stability, also known as abrupt convergence. With the help of the previously established ergodicity bounds, we obtain this property without any additional calculation, other than Hurwitz stability and a genericity assumption of the interaction matrix. As a show case for the complexity of systems which are covered by our theorem, and where explicit calculations are out of question, we study Jacobi chains a more general networks of coupled harmonic oscillators with a fixed amplitude Brownian or L\'evy-type external heat bath forcing. 
}

Since the days of von Smoluchovski \cite{vSmoluchovski1906}, Langevin \cite{Langvin1908}, and Uhlenbeck and Ornstein \cite{OrensteinUhlenbeck1930} more than a century ago and even earlier \cite{Ja17}, the Ornstein-Uhlenbeck process and its extensions to higher and infinite dimensions, and different noises are still 
intensely studied objects in statistical physics, neuronal networks, probability and statistics. Despite their apparent simplicity, and an ever better understanding of them, 
its (multidimensional) dynamics and ergodicity remains an active field of research, see for instance \cite{Ba14,BrzezniakZabczyk,GTR23,GL19,JL14, Lachaud2005, Mikami, SSB23,SaYa, Simon11,SGA18,TL19, WANGLETTERS} and the numerous references therein. 
Among several competing concepts to measure the thermalization of the current state of such systems 
to their respective dynamic equilibria, such as for instance relative entropy, total variation or the Hellinger distance, and others \cite{DLVV21, DL82, MPZ19, Mi22, OP82, ZP19} the WKR distance (see Definition~\ref{def:Wasserstein} below) stands out: due to its statistical robustness \cite{CPP20,CT16, Panaretos, PJDS14, SP23+}, explicit formulas in the Gaussian case, as for instance \cite{BL23, Givens, Ta11},
{
its deep connections to optimal transport and the Monge-Kantorovich problem, and an extensive calculus which allows for many explicit calculations and sharp bounds, see for instance~\cite{BHPWA, BJL19, BJL19b, CPP20, CT23, FG21, Gel, OPV14, PC19, Panaretos, Vi09}.}

{In this paper we quantify the ergodicity in the WKR distance 
for  multidimensional L\'evy driven Ornstein-Uhlenbeck systems with fixed noise amplitude (see formula \eqref{eq:modelo0} and  Section~\ref{s:ergodicitybounds}).
The novelty of our approach in this paper consists in a particular change of perspective of the classical \textit{cutoff phenomenon} (mathematical terminology) or \textit{abrupt thermalization} (physics terminology) for linear systems with additive noise. Essentially the complete mathematical and physics literature on the cutoff phenomenon in discrete time and space, describes the cutoff phenomenon -roughly speaking- as an asymptotic threshold phenomenon for a family of objects parametrized by an \textit{internal parameter} $\e$ of the system, often representing the (inverse) size of the state space, the dimension of the space or for instance as noise amplitude. 
Standard references in this highly active field of research include for instance \cite{BLY06, BY2014,BASU, BD92, B-HLP19, BBF08, BCS19, BCS18, CSC08, HS20, Jonson, LL19, La16, LanciaNardiScoppola12, LLP10, LPW, LubetzkySly13, Meliot14,  Trefethenbrothers, Yc99}, starting with the seminal papers by Diaconis and Aldous on card shuffling \cite{Al83, AD87, AD, DIA96}. 
In the physics literature this concept has received quite some attention recently in the context of quantum Markov chains \cite{Kastoryano12}, chemical reaction kinetics \cite{BOK10}, quantum information processing \cite{Kastoryano13}, statistical mechanics \cite{CS21, LubetzkySly13}, coagulation-fragmentation equations \cite{Pego17, Pego16}, dissipative quantum circuits \cite{JohnssonTicozziViola17}, open quadratic fermionic systems \cite{Vernier20}, neuronal models \cite{dOTL18}, granular flows \cite{WC18}, and chaotic microfluid mixing \cite{LW08}. }

{
In a series of articles \cite{Ba18, BHPWA, BHPWANO, BHPTV,BHPSPDE, BJ, BJ1, BL21, BP} the authors have studied the so-called cutoff phenomenon for abstract Langevin equations with $\e$-small, additive L\'evy noise $\ud L$ (see Definition~\ref{def:defL} below) given by the following stochastic differential equation 
 \begin{equation}\label{eq:modelo00}
\ud X_t(x)=-AX_t(x)\ud t+\e \ud L_t, \quad t\geq 0,\quad \textrm{ with }\quad X_0(x)=x \in \RR^m, x\neq 0, \e>0, 
\end{equation}
with $-A\in \RR^{m\times m}$ being a Hurwitz stable matrix (see Definition~\ref{def:Hurwitz} below) under different kinds of metrics. 
For clarity we introduce various concepts of cutoff phenomena. Assume that there is a parametrized family of processes $X^\e=(X^\e_t(x))_{t\geq 0}$, $\e>0$, of 
invariant measures $\mu^\e$, of renormalized distances $d_\e$ on the space of probability distributions in the state space and a deterministic time scale $t_{\e}$ such that for 
\begin{equation}\label{e:cutoffphenomenon}
D_{\e, x}(t) := d_\e(X^\e_{t}(x), \mu^\e), \qquad x\in \RR^d, x \neq 0, \e>0,    
\end{equation}
we have one of the three cutoff phenomena in the sense of \cite{BY2014}:
\begin{enumerate}
 \item A time scale $(t_\e)_{\e >0}$ induces a \textit{(simple) cutoff phenomenon} if $D_{\e,x}(\delta t_\e)$ tends to the maximal value $M$ of the distance if $\delta < 1$, to $0$ if $\delta > 1$.
 \item A time scale induces a \textit{window cutoff phenomenon} if $\liminf_{\e \ra 0} D_{\e,x}(t_\e + r )$ tends 
to $M$ as $r$ tends to $-\infty$, and $ \limsup_{\e \ra 0} D_{\e,x}(t_\e + r)$ tends to $0$ as $r$ 
tends to $\infty$.
 \item A time scale $(t_\e)_{\e >0}$ induces a \textit{profile cutoff phenomenon} with 
 \textit{cutoff profile} $\mathcal{P}_x$ if $\mathcal{P}_x(r) = \lim_{\e\ra 0} D_{\e,x}(t_\e+ r)$ exists for all $r\in \RR$, 
 and $\mathcal{P}_x$ tends to $M$ at $-\infty$, to $0$ at $\infty$.
\end{enumerate}
The parameter $\e$ in the previously mentioned articles, is the noise intensity in \eqref{eq:modelo00}. In \cite{Ba18,BJ, BJ1, BP, BHPTV} $d_\e = d_{TV}$ the (unnormalized) total variation distance, while in \cite{BHPWA, BHPWANO, BHPSPDE} the distance is given by $d_\e = \wW_p /\e$, the renormalized WKR-distance of order $p\geq 2$. 
}

{
The main idea of this article is based on the following observation. Note that the concepts (1), (2) and (3) for $D_{\e, x}$ defined in \eqref{e:cutoffphenomenon} along a time scale $t_\e$ do not exclude the special case, where $X$ and $\mu$ do \textit{not} depend on the noise intensity parameter $\e$, that is, for \textit{fixed} noise amplitude $\sigma$. 
That is, the object of study of this article is the Ornstein-Uhlenbeck system \eqref{eq:modelo0}, which does not depend on any parameter $\e$ in any sense. 
More precisely, we consider the dynamics of the unique strong solution $X = (X_t)_{t\geq 0}$ of the following 
stochastic differential equation 
 \begin{equation}\label{eq:modelo0}
\ud X_t(x)=-AX_t(x)\ud t+\sigma\ud L_t, \quad t\geq 0,\quad \textrm{ with }\quad X_0(x)=x \in \RR^m,
\end{equation}
towards its dynamic equilibrium distribution $\mu$ on $\RR^m$, where 
$-A\in \RR^{m\times m}$ is Hurwitz stable $\sigma \in \RR^{m\times n}$ and $L$ a $n$-dimensional Brownian motion. More generally, $L$ can be a L\'evy process, such as a compound Poisson process or an $\alpha$-stable L\'evy flight. See for instance \cite{APPLEBAUMBOOK, Mao2008, Protter, Sato}. 
}

{
In this case, finding a particular time scale $(t_\e)_{\e>0}$ such that for 
\begin{equation}\label{e:cutoffstability}
\mathcal{D}_{\e, x}(t) := d_\e(X_t(t), \mu) = \frac{\mathcal{W}_p(X_t(x), \mu)}{\e}
\end{equation}
satisfies the concepts (1)-(3), yields the (asymptotic) reparametrization of the $\e$-smallness of the WKR-distance by $t_\e$ and yields for instance for the concept (1) the threshold phenomenon as $\e \ra 0$ 
\[
\mathcal{W}_p(X_{\delta t_\e}(x), \mu) \begin{cases} \gg \e & \mbox{ if } \delta \in (0,1),\\ \ll \e & \mbox{ if } \delta>1. \end{cases}
\]
The time scale $t_\e$ sharply divides substantially smaller than $\e$-small and substantially larger than $\e$ small values of the distance to the dynamical equilibrium. Due to this $\infty/0$ dichotomy this cutoff phenomenon without internal parameter is called \textit{cutoff stability}. 
We use the notions of \textit{simple cutoff stability}, \textit{window cutoff stability} and \textit{profile cutoff stability} for $\mathcal{D}_{\e, x}(t)$ satisfying (1), (2) or (3), respectively, 
for a time scale $(t_\e)_{\e>0}$. 

In this situation there obviously still appears a parameter $\e>0$ in \eqref{e:cutoffphenomenon}, but in contrast to \eqref{e:cutoffstability}, where it had the role of an \textit{internal} parameter, 
it rather plays the role of an \textit{external} yard stick parameter, which controls the asymptotic WKR mixing times. In \cite{BHPGBM} the authors established such a type of ``nonasymptotic'' cutoff phenomenon for a process with fixed multiplicative noise under certain commutativity conditions. In \cite{BHPPSHELL}, it was established for an infinite dimensional linear energy shell model with scalar random energy injection. This article closes the gap in the literature and studies this concept in the most natural and useful finite dimensional setting with additive noise. 
}

{
We stress that the situation of \eqref{eq:modelo0} is more complicated than the situation of \eqref{eq:modelo00} since it is not quasideterministic, in the sense of being essentially a deterministic system with $\e$-small, though random perturbation. Instead, in \eqref{eq:modelo0} appears a full-blown dynamical equilibrium, which might be rather irregular in the sense of not admitting a density. 
This difficulty is enhanced by the fact that $-A$ is only Hurwitz stable but not diagonalizable in general, which is natural for instance in the case of linear oscillators with friction. Therefore, arbitrarily large Jordan blocks with possibly non-real eigenvalues are permitted, which are present in the limiting distribution. It is one of the advantages of the WKR distance, in comparison to the total variation distance, that it does require any particular regularity beyond the existence of certain moments. 
In particular, it does not exclude degenerate noise injection of the system, such as in the case of the linear oscillator (see Example~\ref{ex:bio} or networks of those Example~\ref{ex:oscJC} and \eqref{ex:oscmoregeneral}). In particular, the WKR distance avoids the technicalities such as controllability associated to the Kalman conditions and hypoellipticity, typically present for results in the total variation distance and the relative entropy, see \cite[Chapter 6]{Pavliotis} and references therein. We consider additive perturbations by multidimensional L\'evy noise processes with first moments, which include Brownian motion, deterministic linear functions, compound Poisson processes and its possibly infinite superposition, such as $\alpha$-stable processes with $1< \alpha\leq 2$, among others. By a standard enhancement of the state space, we also cover the situation of Ornstein-Uhlenbeck noise perturbations with each of the preceding types of noise. 
}

The article is organized in three main parts:  
First, we provide in Theorem~\ref{thm:couplings} of Subsection~\ref{ss:anewformula} the state of the art including new general lower and upper bounds of $\wW_p(X_t(x), \mu)$ of order $p>0$. 
In Subsection~\ref{sub:nonblper} we collect particularly useful Gaussian bounds for $\wW_p$, $p\geq 1$, applied in Subsection~\ref{sec:cutsta}. 

Using the results of Section~\ref{s:ergodicitybounds}, we study \textit{cutoff stability} for systems of the form \eqref{eq:modelo0}. We start with non-degenerate Gaussian systems \eqref{eq:modelo0} for which we use the explicit formulas of Subsection~\ref{ss:anewformula} in order to establish cutoff stability for systems \eqref{eq:modelo0} for the first time in a simple case. More precisely, for normal drift matrix $A$ and non-degenerate dispersion matrix $\sigma$ we provide new explicit formulas for the $\wW_2$-distance in Theorem~\ref{thm:Gauss}, which then imply cutoff stability in the sense of \eqref{e:cutoffstability}. 
In Example~\ref{ex:oscBM} we continue with the study of the scalar damped harmonic oscillator subject to moderate Brownian forcing, which has a degenerate dispersion matrix $\sigma$ in the product space of position and momentum and which is not covered by the formulas in Theorem~\ref{th:OUGcut}. We establish the presence of cutoff stability \eqref{e:cutoffstability} for this elementary, though degenerate, system by explicit calculations, which illustrate the remarkable level of complexity and the infeasibility in general to stick to explicit calculations even for {linear} 2-d Gaussian systems. 

In Theorem~\ref{thm:cutoffWp} of Subsection~\ref{sec:32} we show that the non-asymptotic bounds Theorem~\ref{thm:couplings} in Subsection~\ref{ss:anewformula} are good enough to establish cutoff stability \eqref{e:cutoffstability} in considerably greater generality than Theorem~\ref{thm:Gauss}. Theorem~\ref{thm:cutoffWp} directly covers Example~\ref{ex:oscBM}, the Brownian gyrator in Example~\ref{ex:gyrator}, a biophysical transcription-translation linear oscillator model in Example~\ref{ex:bio},
and
the benchmark system of a Jacobi chain of oscillators with a heat bath of constant noise intensity on the outer edges in Example~\ref{ex:oscJC}.
 More precisely, in Theorem~\ref{thm:cutoffWp} we establish cutoff stability under general $\sigma$ and generic assumptions on $A$, which are substantially weaker than the results in Section~\ref{sec:cutsta}. 
In particular, they include Hurwitz stable, but non-normal interaction matrices $A$, a possibly degenerate dispersion matrix $\sigma$ and a large class of L\'evy drivers, including Brownian motion and $\alpha$-stable L\'evy flights for $\alpha>1$. In Example~\ref{ex:oscmoregeneral} we comment on the validity of our results for more general networks topology.

In Appendix~\ref{a:WKRproperties} the reader finds {a list of }the most relevant properties of the WKR-distances.  

\bigskip
\section{\textbf{Non-asymptotic ergodicity estimates for the multidimensional OU process}}\label{s:ergodicitybounds}

\noindent In this section we show non-asymptotic ergodicity bounds 
for solutions of the system \eqref{eq:modelo0} under the following hypotheses. 

\begin{hypothesis}[Positivity]\label{hyp:stable}
The matrix $A\in \mathbb{R}^{m\times m}$ is constant and
all its eigenvalues have strictly positive real parts.
\end{hypothesis}

\begin{definition}\label{def:Hurwitz}
A matrix such that $-A$ satisfies Hypothesis~\ref{hyp:stable} is called Hurwitz stable. 
 \end{definition}

\begin{hypothesis}[Diffusion matrix]\label{hyp:diff} The matrix $\sigma\in \mathbb{R}^{m\times n}$ is constant.
\end{hypothesis}

\noindent {We stress that Hypothesis~\ref{hyp:diff} on our model \eqref{eq:modelo0} states that the diffusion matrix is \textit{fixed} and non-small. In fact, there is no particular parameter dependence whatsoever. } 
For convenience we formulate the following elementary lemma {for} Hurwitz stable matrices. 

\begin{lemma}\label{lem:Hurwitz}
Let $A, B\in \mathbb{R}^{m\times m}$ be Hurwitz stable matrices. Then we have the following: 
\begin{enumerate}
 \item $A$ is invertible and $A^{-1}$ is Hurwitz. 
 \item If $AB = BA$, then $A+B$ is Hurwitz stable. If $AB \neq BA$, there are counterexamples. 
\end{enumerate}
\end{lemma}

\noindent The proof of Lemma~\ref{lem:Hurwitz} is given in Appendix~\ref{a:Hurwitz}. There is a large literature 
on the respective matrix theory we refer to \cite{BKR17, Bh07, LT85}.

\begin{definition}[WKR distance of order $p>0$]\label{def:Wasserstein}
For probability distributions $\mu_1,\mu_2$ on $\mathbb{R}^m$ 
with finite $p$-th moments, $p>0$, the WKR distance $\wW_p$  of order $p>0$ is defined by
\begin{equation}
\Wp(\mu_1,\mu_2):=\left(\inf\limits_{\bT} \int_{\mathbb{R}^m\times \mathbb{R}^m} |u-v|^p \bT(\ud u,\ud v)\right)^{\min\{1,1/p\}},
\end{equation}
where $\bT$ is any joint distribution between $\mu_1$ an $\mu_2$, that is, 
\[\bT(A\times \mathbb{R}^m)=\mu_1(A)\quad \mbox{ and }\quad \bT(\mathbb{R}^m\times A)=\mu_2(A)\qquad \mbox{ for all Borel-measurable sets }A \subset \mathbb{R}^m.\]
\end{definition}
\noindent {The main basic properties of the WKR-distance are gathered in Lemma~\ref{lem:properties}. For more details see \cite{Panaretos, Vi09}.}

\noindent For convenience of notation we do not distinguish a random variable $X$ and its law $\mathbb{P}_X$ as an argument of $\mathcal{W}_p$. That is, for random variables $X$, $Y$ and probability measure $\mu$ we write $\mathcal{W}_p(X, Y)$ instead of $\mathcal{W}_p(\mathbb{P}_{X}, \mathbb{P}_{Y})$, $\mathcal{W}_p(X, \mu)$ instead of $\mathcal{W}_p(\mathbb{P}_{X}, \mu)$ etc.

{
\subsection{\textbf{A formula for the WKR-2-distance}}\hfill\label{ss:anewformula}\\
}
\smallskip 

{
Denote by $\mathcal{N}(m, C)$ the $m$-dimensional normal distribution with expectation $m$ and covariance matrix $C$. For a square matrix $C=(C_{i,j})\in \mathbb{R}^{m\times m}$ 
we denote its trace by $\textsf{Tr}(C):=\sum_{j=1}^m c_{j,j}$. For any matrix $M$ with real coefficients we denote by $M^*$ its transpose, while for any matrix $M$ with complex coefficients, $M^*$ denotes the Hermitian transpose.}

{
Using the translation invariance given in Item~b) in Lemma~\ref{lem:properties} we have 
\begin{equation}
\begin{split}
\mathcal{W}_2(X_t(x),\mu)&=\mathcal{W}_2(\mathcal{N}(e^{-At}x+A^{-1}(I_m-e^{-At})\sigma b, \Sigma_t),\mathcal{N}(A^{-1}\sigma b,\Sigma_\infty))\\
&
=\mathcal{W}_2(\mathcal{N}(e^{-At}x-A^{-1}e^{-At}\sigma b, \Sigma_t),\mathcal{N}(0,\Sigma_\infty))\\
&=\mathcal{W}_2(\mathcal{N}(e^{-At}(x-A^{-1}\sigma b), \Sigma_t),\mathcal{N}(0,\Sigma_\infty)), 
\end{split}
\end{equation}
where
\begin{align*}
\Sigma_t &= \int_0^t e^{-As}\sigma \sigma^* e^{-A^*s} \ud s \qquad \mbox{ and }\qquad \Sigma_\infty = \int_0^\infty e^{-As}\sigma \sigma^* e^{-A^*s} \ud s.
\end{align*}
}

{
 \noindent We show an exact formula of the WKR-distance of order $2$ between a standard multidimensional OU process (with $\sigma\sigma^* = I_m$ and $L$ a standard Brownian motion in $\RR^m$), and its invariant measure $\mu = \mathcal{N}(0, \Sigma_\infty)$, see Remark~\ref{rem:commnormal} (3), which we are not aware of in the literature. 
 \bigskip
\begin{theorem}[$\wW_2$-ergodicity formula for normal interaction matrices and full Brownian forcing]\label{th:OUGcut}\hfill
\\
\noindent
Assume that $\sigma\sigma^* = I_m$.
\begin{enumerate}
 \item If $A$ is a positive definite symmetric matrix with eigenvalues $0<\lambda_1\leq \lambda_2\leq \cdots\leq \lambda_m$ and corresponding orthogonal eigenvectors $v_1,\ldots,v_m$, then for any $x\in \mathbb{R}^m$ and $t\geq 0$ it follows that
\begin{equation}\label{e:PitagorasS}
\mathcal{W}_2(X_t(x),\mu)=
\left(\sum_{j=1}^{m} e^{-2\lambda _j t}\langle x-A^{-1}\sigma b,v_j \rangle^2+
\sum_{j=1}^{m}\frac{1}{2\lambda_j}\frac{e^{-4\lambda_j t}}{(\sqrt{1-e^{-2\lambda_j t}}+1)^2}\right)^{1/2}.
\end{equation}
 \item If $A$ is a normal matrix $A$ that is, $A A^* = A^* A$, and $A+A^*$ has the following eigenvalues ordered by  $0<\varphi_1\leq \varphi_2\leq \cdots\leq \varphi_m$ and corresponding {(generalized)} orthonormal eigenvectors $v_1,\ldots,v_m\in \mathbb{R}^m$, then
for any $x\in \mathbb{R}^m$ and $t\geq 0$ it follows that
\begin{align}\label{e:Pitagoras}
&\mathcal{W}_2(X_t(x),\mu)=\left(\sum_{j=1}^{m} e^{-\varphi_j t}\langle x-A^{-1}\sigma b,v_j\rangle^2+
\sum_{j=1}^m \frac{1}{\varphi_j} \frac{e^{-2\varphi_j t}}{(\sqrt{1-e^{-\varphi_j t}}+1)^2}\right)^{1/2},
\end{align}
where $\varphi_j=2\textsf{Re}(\la_j)$, $j=1,\ldots,m$ and $\la_j$, are the eigenvalues of $A$ (ordered in ascending by its real parts).
 \end{enumerate}
\end{theorem}
}

{
\begin{remark}\label{rem:explan}
\begin{enumerate}
 \item The main insight from formulas \eqref{e:PitagorasS} and \eqref{e:Pitagoras} 
is that the WKR-2 distance (implicitly due to the Pythagorean theorem) naturally reflects the dynamics of the mean 
and the variance of the Ornstein-Uhlenbeck process. 
In case of $m=n= 1$ we have for the solution of 
\begin{align*}
\ud X_t(x) = - \la X_t(x) \ud t + \ud B_t, \qquad X_0(x) = x,
\end{align*}
that the limiting distribution is $\nu = \mathcal{N}(0, \frac{1}{2\lambda})$ and 
\[
\EE[X_t(x)] = e^{-\lambda t} x
\qquad \mbox{ and }\qquad \mbox{Var}(X_t(x))  = \frac{1}{2\lambda} (1- e^{-2 \lambda t}) 
\]
that is, the variance adjusts to the limiting variance $\frac{1}{2\lambda}$ at double the speed, 
than the mean converges to $0$ in the limit.   
\item In the case of L\'evy drivers we observe that a $m$-dimensional pure jump L\'evy process $L$ 
cannot be generically decomposed by a sort of principal axes transform just as multivariate Brownian motion in a vector of independent scalar L\'evy processes  
\[
L = (L^1, \dots, L^m). 
\]
Clearly, such L\'evy processes do exist but they only refer to L\'evy flights 
with jumps parallel to the axes, which is a very special subcase of limited interest, see~\cite{Sato}. 
\\
\item We conjecture the mean versus variance separation of scales of item (2), 
to be true for all L\'evy processes with second moments. 
Let $L=(L_s)_{s\geq 0}$ be a symmetric $\alpha$-stable process with $0<\alpha\leq 2$. More precisely, the  characteristic function of the marginal at time $t\geq 0$, $L_t$, is given by $\mathbb{E}[e^{\textsf{i} z L_t}]=e^{-t|z|^\alpha}$, $z\in \mathbb{R}$. By
Lemma~17.1. in \cite{Sato} for the Ornstein-Uhlenbeck process
$X_t=e^{-\lambda t}x+\sigma e^{-\lambda t}\int_{0}^{t}e^{\lambda s}\ud L_s$
it follows that the characteristic function of 
$X_t$ is given by
\begin{equation}
\begin{split}
\mathbb{R}\ni z\mapsto\mathbb{E}[e^{\textsf{i} z X_t}]&=\exp\left(\textsf{i}e^{-\lambda t}xz+\int_{0}^t |e^{-\lambda s}\sigma z|^\alpha \ud s\right)\\
&=\exp\left(\textsf{i}e^{-\lambda t}xz+\sigma^\alpha\frac{1-e^{-\lambda \alpha t}}{\lambda \alpha}|z|^\alpha \right),
\end{split}
\end{equation}
which yields that 
$X_t=e^{-\lambda t}x+\sigma\left(\frac{1-e^{-\lambda \alpha t}}{\lambda \alpha}\right)^{1/\alpha}L_1$, where the equality is in distribution sense. Hence, the invariant measure $\mu$ has law $\sigma\left(\frac{1}{\lambda \alpha}\right)^{1/\alpha}L_1$.
Therefore, for $1<p<\alpha$ it follows that
\begin{equation}
\begin{split}
\mathcal{W}_p(X_t,\mu)&=\mathcal{W}_p\left(e^{-\lambda t}x+\sigma\left(\frac{1-e^{-\lambda \alpha t}}{\lambda \alpha}\right)^{1/\alpha}L_1,\sigma\left(\frac{1}{\lambda \alpha}\right)^{1/\alpha}L_1\right)\\
&\leq \left(\mathbb{E}\left[|e^{-\lambda t}x+\frac{\sigma}{(\lambda \alpha)^{1/\alpha}}(1-(1-e^{-\lambda \alpha t})^{1/\alpha})L_1|^p\right]\right)^{1/p}.
\end{split}
\end{equation}
We see, that for $x\neq 0$ (or more generally $x \neq \lambda^{-1}\sigma \EE[L_1]$, see Remark~\ref{rem:recentrar}), 
the convergence of the right-hand side to $0$ as $t\ra\infty$ is of order $e^{-\lambda t}$. 
However, starting precisely in $x = 0$ we obtain due to the Taylor expansion of 
\[
1- (1-y^\alpha)^\frac{1}{\alpha} = \frac{1}{\alpha} x^\alpha + O(x^{2\alpha})_{x\ra 0}
\]
the accelerated asymptotic rate $e^{-\lambda \alpha t}$ as $t\ra\infty$.\\ 
\item 
In higher dimensions there are no general known explicit formulas for the WKR-2 distance (or any other WKR-$p$ distance) between non-Gaussian distributions. For one dimensional formulas, see for instance Section~3 in \cite{GHKM17} and the references therein. 
That is, one is sent back to the original optimization over all couplings (or replica). 
Optimizers, so-called, optimal couplings are unknown, which is why, the general case for multidimensional L\'evy drivers with second moments seems hard to prove.\\
With no identities for the optimal coupling at hand, 
we can only prove suboptimal upper bounds, as given in Theorem~\ref{thm:couplings} and Theorem~\ref{thm:Gauss}, which cannot distinguish the mean-variance split of item (3). 
These results, however, hold for general WKR-$p$ distances, $p\geq 1$, and are not restricted to order $p =2$. We note that in the non-Gaussian case even these new suboptimal lower and upper bounds are not straightforward. In particular, 
we stress that lower bounds are typically hard to obtain. 
While these estimates will not allow for a fine properties such \textit{profile cutoff stability} (see item (3) in the introduction), but still the weaker property of \textit{simple cutoff stability} and \textit{window cutoff stability}. 
\end{enumerate}
\end{remark}
}

{
\begin{proof}[Proof of Theorem~\ref{th:OUGcut}]
It is enough to show the case of $A$ being a normal matrix $A A^* = A^* A$. 
Recall that $A\in \mathbb{R}^{m\times m}$ and $\sigma\in \mathbb{R}^{m\times n}$. Note that  
$A$ and $A^*$ have the same eigenvalues, which we denote by $\lambda_j$, $1\leq j\leq m$. 
The Pythagorean Theorem \cite[Proposition~7]{Givens} yields
\begin{align}
\mathcal{W}_2(\mathcal{N}(e^{-At}x,\Sigma_t),\mathcal{N}(0,\Sigma_\infty)) 
& = 
 \left(
|e^{-At }\widetilde{x}|^2+ \textsf{Tr}\left({\Sigma_t+\Sigma_\infty-2(\Sigma^{1/2}_t\Sigma_\infty \Sigma^{1/2}_t)^{1/2}}\right)\right)^{1/2}\nonumber\\
&= \left(
|e^{-At }\widetilde{x}|^2+ \textsf{Tr}\left({\Sigma_t+\Sigma_\infty-2(\Sigma^{1/2}_\infty \Sigma_t \Sigma^{1/2}_\infty)^{1/2}}\right)\right)^{1/2}, \label{e:Pythagoras}
\end{align}
where $\widetilde{x}:=(x-A^{-1}\sigma b)$, 
\begin{align*}
\Sigma_t &= \int_0^t e^{-As}\sigma \sigma^* e^{-A^*s} \ud s \qquad \mbox{ and }\qquad \Sigma_\infty = \int_0^\infty e^{-As}\sigma \sigma^* e^{-A^*s} \ud s.
\end{align*}
By hypothesis $\sigma \sigma^*=I_m$.
By the Baker-Campbell-Hausdorff-Dynkin formula \cite[Chapter 5]{Ha15} we have 
\begin{align*}\label{eq:fhg1}
\Sigma_t &= \int_0^t e^{-As} e^{-A^*s} \ud s = \int_0^t e^{-s (A+A^*)} \ud s 
=  (A+A^*)^{-1} \Big(I_m - e^{-t(A+A^*)}\Big). 
\end{align*}
By Item~(2)~in Lemma~\ref{lem:Hurwitz} it follows that $(-A)+(-A^*)$ is Hurwitz stable.
Therefore, $\lim_{t\to \infty}\Sigma_t=B^{-1}=:\Sigma_\infty$, where $B:=A+A^*$. 
Since $B$ is a symmetric matrix, the
eigenvalues $\varphi_1,\ldots,\varphi_m$ of $B$ are real numbers. Moreover, without loss of generality we can assume that $0<\varphi_1\leq \cdots\leq \varphi_m$.
Hence, 
\begin{align*}
 &\textsf{Tr}(\Sigma_t) =  \sum_{j=1}^m \frac{1}{\varphi_j} \Big(1- e^{-\varphi_j t}\Big)\qquad\mbox{ and }\qquad \textsf{Tr}(\Sigma_\infty ) = \sum_{j=1}^m \frac{1}{\varphi_j}. \end{align*}
On the other hand, we have 
\begin{align*}
\Sigma_\infty^\frac{1}{2} \Sigma_t \Sigma_\infty^\frac{1}{2} 
&= B^{-\frac{1}{2}}B^{-1}(I_m -e^{-tB}) B^{-\frac{1}{2}} \\
&=
B^{-\frac{1}{2}}B^{-1}B^{-\frac{1}{2}}(I_m -e^{-tB}) 
=
B^{-2} (I_m - e^{- t B}), 
\end{align*}
which implies by the spectral calculus theorem 
\begin{align*}
&-2\,\textsf{Tr}\left((\Sigma_\infty^\frac{1}{2} \Sigma_t \Sigma_\infty^\frac{1}{2} )^\frac{1}{2}\right) 
 = -2\sum_{j=1}^m \frac{1}{\varphi_j} \sqrt{1- e^{- \varphi_j t }}, 
\end{align*}
and hence 
\begin{align*}
\textsf{Tr}\Big( \Sigma_t + \Sigma_\infty - 2 (\Sigma_\infty^\frac{1}{2} \Sigma_t \Sigma_\infty^\frac{1}{2})^\frac{1}{2}\Big)
&=  \sum_{j=1}^m \Big(\frac{1}{\varphi_j} (1-e^{-\varphi_j t}) + \frac{1}{\varphi_j} - 2\frac{1}{\varphi_j} \sqrt{1- e^{- \varphi_j t })}\Big)\\
&=  \sum_{j=1}^m \left( \sqrt{\frac{1}{\varphi_j}}-\sqrt{\frac{1}{\varphi_j} (1- e^{-\varphi_j t})}\right)^2\\
&=  \sum_{j=1}^m \frac{1}{\varphi_j} \frac{e^{-2\varphi_j t}}{(\sqrt{1-e^{-\varphi_j t}}+1)^2}.
\end{align*}
Again, the Baker-Campbell-Hausdorff-Dynkin formula gives
\[
|e^{-tA}\widetilde{x}|^2=\widetilde{x}^* e^{-tA^*}e^{tA}\widetilde{x}=\widetilde{x}^*e^{-Bt}\widetilde{x}=\sum_{j=1}^{m} e^{-\varphi_j t}y^2_j,
\]
where $y=O\widetilde{x}$, and $O^*DO=B$ with $O=(O_{i,j})_{i,j\in \{1,\ldots,m\}}$ being an orthogonal matrix in $\mathbb{R}^{m\times m}$, $D=\textsf{diag}(\varphi_1,\ldots,\varphi_m)$ being a diagonal matrix in $\mathbb{R}^{m\times m}$ having in the diagonal the eigenvalues of $B$. More precisely, we have 
$y_i=\langle O_i,\widetilde{x}\rangle$, where $O_i=(O_{i,j})_{j\in\{1,\ldots,m\}}$.
\\
In the sequel, we calculate $\varphi_j$, $j=1,\ldots, m$.
Since $A$ is a normal matrix, we have that
$A=U^*DU$, where $UU^*=U^*U=I_m$.
Recall that $A\in \mathbb{R}^{m\times m}$.
Then $A^T=A^*=U^*D^*U$, where $T$ denotes the transpose.
Thus, $A+A^T=U^*(D+D^*)U$, yields that the eigenvalues of 
$A+A^*$ are $2\textsf{Re}(\la_j)$, $j=1,\ldots,m$ where 
$\la_j$, $j=1,\ldots,m$ are the eigenvalues of $A$.
\\
This completes the proof. 
\end{proof}
}

\bigskip 
\subsection{\textbf{Hypotheses on the non-Brownian L\'evy perturbations}}\label{sub:nonblper}

\begin{definition}[L\'evy noise]\label{def:defL} The driving noise  $L=(L_t)_{t\geq 0}$ is a L\'evy process in $\RR^n$, that is, a stochastic process starting in $0\in \mathbb{R}^n$  with stationary and independent increments, and right-continuous paths (with finite left limits).
\end{definition}

\begin{remark}~
\begin{enumerate}
    \item The class of L\'evy processes $L$ contains several cases of interest: (1) $n$-dimensional standard Brownian motion, (2) $n$-dimensional symmetric and asymmetric $\alpha$-stable L\'evy flights, (3) $n$-dimensional compound Poisson process, (4) deterministic linear function $t\mapsto \gamma t$, $\gamma\in \mathbb{R}^n$.
    \item Under (2) and (3), the paths contain jump discontinuities. Furthermore, the existence of right-continuous paths with left limits (for short RCLL or c\`adl\`ag from the French ``continue \`a droite, limite \`a gauche") is not strictly necessary and it can be always inferred up to zero sets of paths. 
\end{enumerate}
\end{remark}

\noindent For each probability space $(\Omega,\mathcal{A},\mathbb{P})$, which carries $L$
Hypotheses~\ref{hyp:stable} and \ref{hyp:diff} imply the existence and path-wise uniqueness of the equation \eqref{eq:modelo0} given by
\begin{equation}\label{eq:repre}
X_t(x)=e^{-At}x+X_t(0)\quad \textrm{with}\quad X_t(0):=e^{-At}\int_{0}^t e^{As}\sigma\ud L_s
,\quad t\geq 0,\, x\in \mathbb{R}^m,
\end{equation}
where $e^{Mt}:=\sum_{k=0}^{\infty}M^k t^k/k!$ for any $M\in \mathbb{R}^{m\times m}$ and $t\in \mathbb{R}$. 
{
\begin{remark}\label{rem:recentrar}
When $L$ has at least first moment,
we point out that $L$ needs not be centered in general, however, by the L\'evy property of stationary and independent increments (see Definition~\ref{def:defL} below) it follows that
$L_t=\widetilde{L}_t+bt$ a.s., where $b\in \mathbb{R}^m$, 
$\widetilde{L}=(\widetilde{L}_t)_{t\geq 0}$ is a centered L\'evy process. In other words, the mean of \eqref{eq:modelo0} and its limiting distribution are not necessarily centered at the origin, but in $A^{-1}(I_m-e^{-At})\sigma b$ and $A^{-1}\sigma b$, respectively. All our results are valid for any $b\in \mathbb{R}^m$.
\end{remark}
}

We denote by $|\cdot|$ the norm induced by the standard Euclidean inner product $\langle \cdot, \cdot \rangle$ in $\RR^m$. Moreover, we use the standard Frobenius matrix norm $\|M\|^2 = \sum_{i, j}M_{i,j}^2$, $M\in \RR^{m\times n}$. We denote the mathematical expectation over $(\Omega, \mathcal{A}, \mathbb{P})$ by $\EE$. 

\noindent The following hypothesis is necessary and sufficient to provide the existence of a limiting measure. 
\begin{hypothesis}\label{hyp:logmoments}
The time one marginal of $L$ satisfies $\mathbb{E}[\log(1+|L_1|)]<\infty$.
\end{hypothesis}
\noindent Note that Hypothesis~\ref{hyp:logmoments} includes Brownian motion, all $\alpha$-stable L\'evy flights and compound Poisson processes where the jump measure has a finite logarithmic moments. 
We point out that under Hypotheses~\ref{hyp:stable}, \ref{hyp:diff} and \ref{hyp:logmoments} there is a unique stationary probability distribution $\mu$ for the random dynamics \eqref{eq:modelo0}. Moreover, for any initial data $x\in \mathbb{R}^m$, $X_t(x)$ converges in distribution to $\mu$ as $t\to \infty$, see for instance \cite{SaYa, WangJAP} and \cite{KS14} for the Gaussian case.

\bigskip
\subsection{\textbf{Ergodicity bounds via disintegration for $\wW_p$, $p\geq 1$}}\hfill\\

\noindent In order to measure the convergence towards the dynamic equilibrium by $\wW_p$, $p\geq 1$, 
we assume the following stronger condition than Hypothesis~\ref{hyp:logmoments}.
\begin{hypothesis}[Finite moment]\label{hyp:moments}\hfill\\
There is $p>0$ such that $\mathbb{E}[|L_1|^p]<\infty$.
\end{hypothesis}
\begin{remark}
Note that Hypothesis~\ref{hyp:moments} yields 
$\mathbb{E}[|L_t|^p]<\infty$ and
$\mathbb{E}[|X_t(x)|^p]<\infty$ for any $t\geq 0$ and $x\in\mathbb{R}^m$. 
\end{remark}

\noindent Since the convergence in $\wW_p$ is equivalent to the convergence in distribution and the {simultaneous} convergence of the $p$-th absolute moments we have to ensure that the thermalization coming from Hypothesis~\ref{hyp:logmoments} also holds in the stronger WKR sense.  

\begin{lemma}[Ergodicity in $\wW_p$] Assume Hypotheses~\ref{hyp:stable}, \ref{hyp:diff} and \ref{hyp:moments} for some $p>0$. Then there is a unique probability measure $\mu$ in $\mathbb{R}^m$ such that for all $x\in \mathbb{R}^m$
\[
\lim\limits_{t\to \infty}\Wp(X_t(x), \mu)=0.
\]
In particular, it is stationary for \eqref{eq:modelo0}, that is, for all $t\geq 0$,  $X_t(\mu)=\mu$ in distribution.
\end{lemma}

\noindent This result is shown in \cite[Proposition 2.2]{Ma04}. 
By  \cite[Proposition 2.2]{Ma04} Hypotheses ~\ref{hyp:stable}, \ref{hyp:diff} and \ref{hyp:moments}  {imply} the existence of a unique equilibrium distribution $\mu$, and its statistical characteristics such as $p$-th moments are given there. 

We now formulate the first main result on the ergodicity bounds for the marginal of $X$ at time $t$. 

 \begin{theorem}[Quantitative ergodicity bounds for L\'evy driven Ornstein-Uhlenbeck systems]\label{thm:couplings}\hfill\\
  Assume the Hypotheses~\ref{hyp:stable}, \ref{hyp:diff} and \ref{hyp:moments} for some $p>0$. 
Then we have for all $t\geq 0$, {$x\in \RR^m$} the following bounds: 
\begin{enumerate}
\item Upper bounds: 
\begin{align*}
\mathcal{W}_p(X_t(x),\mu) \leq 
\begin{cases} 
|e^{-At}x|+\mathcal{W}_p(X_t(0),\mu), & \\[2mm]
\int_{\mathbb{R}^m} |e^{-At}(x-y)|^{\min\{1,p\}}\mu(\ud y), &
\end{cases} 
\end{align*}
and, in particular, 
\[
\mathcal{W}_p(X_t(0),\mu) \leq  \int_{\mathbb{R}^m} |e^{-At}y|^{\min\{1,p\}}\mu(\ud y).
\]
\item Lower bounds:
\begin{align*}
\mathcal{W}_p(X_t(x),\mu) \geq 
\begin{cases}
|e^{-At}x|-\mathcal{W}_p(X_t(0),\mu) & \mbox{if }p\geq 1,\\[2mm]
\left|e^{-At}x+\mathbb{E}[X_t(0)]-\int_{\mathbb{R}^m} z \mu(\ud z ) \right| & \mbox{if }p\geq 1,\\[2mm]
|e^{-At}x|^p-2\mathbb{E}[|X_t(0)|^p]-\mathcal{W}_p(X_t(0),\mu) & \mbox{if }p\in (0,1),\\[2mm]
0 & \mbox{if }p>0,
\end{cases}
\end{align*}
where for the identity matrix $I_m\in \RR^{m\times m}$ we have 
\[
\mathbb{E}[X_t(0)]=e^{-At} \int_{0}^{t} e^{As}\sigma \mathbb{E}[L_1] \ud s = A^{-1} (I_m- e^{-At}) \sigma \mathbb{E}[L_1].
\]
\end{enumerate}
 \end{theorem}
 \noindent The proof is given in Appendix~\ref{A:proofcouplings}. It heavily draws on the properties of the WKR distance gathered in Lemma~\ref{lem:properties} of Appendix~\ref{a:WKRproperties}. 
 By Jensen's inequality we have $\mathbb{E}[|X_t(0)|^p]\leq \mathbb{E}[|X_t(0)|]^p$ for $p\in (0,1]${. An} upper bound of $\mathbb{E}[|X_t(0)|]$ is given in \cite[p. 1000-1001]{WangJAP}.

\bigskip
\subsection{\textbf{Ergodicity bounds via Gaussian estimates for $\mathcal{W}_p$, $p\geq 2$}}\hfill\\

\noindent It is remarkable, that under many circumstances, that is, for $p\geq 2$, meaningful Gaussian estimates can be given for WKR-distances of order $p\geq 2$ between general non-Gaussian L\'evy-OU processes and their equilibrium, in the following sense.   

 \begin{theorem}[Gaussian ergodicity bounds for non-Brownian, L\'evy Ornstein-Uhlenbeck systems]
 \label{thm:Gauss}\hfill\\ 
 Let Hypothesis~\ref{hyp:stable}, \ref{hyp:diff} and \ref{hyp:moments} be satisfied for some $p\geq 2$. Then for all $t\geq 0$ and $x\in \mathbb{R}^m$ it follows 
\begin{equation}\label{eq:lowerupper}
\begin{split}
&\Big(
   |e^{-At}x|^2 
    + 
   \textsf{Tr}\left(
    \Sigma_t+\Sigma_\infty -2(\Sigma^{1/2}_t \Sigma_\infty \Sigma^{1/2}_t)^{1/2}
 \right)
 \Big)^{1/2}\\[2mm] 
 &\qquad = \mathcal{W}_2(\mathcal{N}(e^{-At}x,\Sigma_t),\mathcal{N}(0,\Sigma_\infty))\nonumber\\[4mm]
 &\qquad \qquad \leq \mathcal{W}_p(X_t(x),\mu)\\
 &\qquad \qquad \qquad  \leq |e^{-At}x|+ 
  \left(\mathbb{E}[|\int_{t}^{\infty}  e^{-Ar}\sigma\ud L_r|^p]\right)^{1/p},
\end{split}
  \end{equation}
\begin{equation}\label{eq:defsigma}
\mbox{where }\qquad \Sigma_t:=\int_{0}^{t} e^{-A s}\sigma \sigma^* e^{-A^* s} \ud s,\qquad\mbox{ and }\qquad 
\Sigma_\infty:=\int_{0}^{\infty} e^{-A s}\sigma \sigma^* e^{-A^* s} \ud s.
\end{equation}
\end{theorem}

\begin{proof}[Proof of Theorem~\ref{thm:Gauss}]
We start with the proof of the most right inequality of \eqref{eq:lowerupper}. 
{Integration by part implies} 
\begin{equation}\label{e:OUlaw}
\begin{split}
X_t(0)&=e^{-At}
\left(e^{At}\sigma L_t-\int_{0}^{t}Ae^{As}\sigma L_s \ud s \right)=\sigma L_t-\int_{0}^{t}Ae^{-A(t-s)}\sigma L_s \ud s\\
&=\sigma L_t-\int_{0}^{t}Ae^{-Ar}\sigma L_{t-r} \ud r=
\sigma L_t-\int_{0}^{t}Ae^{-Ar}\sigma(L_{t-r}-L_t) \ud r -
\int_{0}^{t}Ae^{-Ar} \ud r \sigma L_t\\
&=\sigma L_t-\int_{0}^{t}Ae^{-Ar}\sigma L_{t-r} \ud r=
\sigma L_t+\int_{0}^{t}Ae^{-Ar}\sigma(L_t-L_{t-r}) \ud r +(e^{-At}-I_m)\sigma L_t\\
&=\int_{0}^{t}Ae^{-Ar}\sigma (L_t-L_{t-r}) \ud r +e^{-At}\sigma L_t\\
&\stackrel{d}{=}\int_{0}^{t}Ae^{-Ar}\sigma \tilde{L}_r \ud r +e^{-At}\sigma\tilde{L}_t=\int_{0}^{t}e^{-Ar}\sigma\ud \tilde{L}_r.
\end{split}
\end{equation}
By the definition of $\mathcal{W}_p$ we have
\begin{equation}
\begin{split}
\mathcal{W}_p(X_t(x),\mu) & \leq 
\left(\mathbb{E}\Big[\big|(e^{At}x+\int_{0}^{t}e^{-Ar}\sigma\ud  \tilde{L}_r)-\int_{0}^{\infty}e^{-Ar}\sigma\ud \tilde{L}_r\big|^p\Big]\right)^{1/p}\\
&\leq |e^{At}x|+\left(\mathbb{E}\Big[|\int_{t}^{\infty}e^{-Ar}\sigma\ud \tilde{L}_r|^p\Big]\right)^{1/p},\end{split}
\end{equation}
where in the last inequality we used Minkowski's inequality for $L^p$.
This proves the most right inequality of \eqref{eq:lowerupper}. 
We continue with the proof of the remaining inequalities. By Jensen's inequality we have
\[
\mathcal{W}_2(X_t(x),\mu)\leq 
\mathcal{W}_p(X_t(x),\mu)
\]
for all $p\geq 2$. By \cite[Theorem~2.1]{Gel} we obtain
\[
\mathcal{W}_2(\mathcal{N}(e^{-At}x,\Sigma_t),\mathcal{N}(0,\Sigma_\infty))\leq \mathcal{W}_2(X_t(x),\mu).
\]
Finally, the Pythagorean Theorem \cite[Proposition~7]{Givens} yields
\begin{equation}\label{eq:formulagauss}
\mathcal{W}_2(\mathcal{N}(m^x_t,\Sigma_t),\mathcal{N}(0,\Sigma_\infty))=\left(
|e^{-At}x|^2+ \textsf{Tr}\left({\Sigma_t+\Sigma_\infty-2(\Sigma^{1/2}_t\Sigma_\infty \Sigma^{1/2}_t)^{1/2}}\right)\right)^{1/2}.
\end{equation}
This completes the proof.
\end{proof}

\begin{remark}\label{rem:tracita}
\begin{enumerate}
 \item By the Pythagorean Theorem given in \cite[Proposition~7]{Givens}
it is clear (consider $x=0$) that
\begin{equation}\label{eq:tracita}
\textsf{Tr}\left({\Sigma_t+\Sigma_\infty-2(\Sigma^{1/2}_t\Sigma_\infty \Sigma^{1/2}_t)^{1/2}}\right)\geq 0,
\end{equation}
and hence for all $t\geq 0$ and $x\in \mathbb{R}^m$ it follows the smaller lower bound 
$|e^{-At}x |\leq\mathcal{W}_p(X_t(x),\mu)$. 
Since the preceding trace terms are hard to calculate, we give upper bounds for $p=2$, 
which are easier to obtain, and which turn out to be sharp whenever $A$ is a normal matrix 
(see Remark~\ref{rem:commnormal}). 
 \item Note that for a pure jump L\'evy process $L$ with finite second moment (see \cite{APPLEBAUMBOOK, Sato}) and $p=2$ we have by It\^o's isometry 
\begin{equation}
\begin{split}
\mathbb{E}[|\int_{t}^{\infty}e^{-Ar}\sigma\ud L_r|^2]= \int_{t}^{\infty}\int_{|z|<1}|e^{-Ar}\sigma z|^2 \nu(\ud z) \ud r,
\end{split}
\end{equation}
where $\nu$ is the L\'evy jump measure associated to $L$, see \cite{APPLEBAUMBOOK, Sato}.
\end{enumerate}
\end{remark}

\begin{corollary}\label{cor:GaussW2}
Let the hypotheses of Theorem~\ref{thm:Gauss} be satisfied for $p=2$. If $L=B=(B_1,\ldots,B_m)$ is a standard Brownian motion in $\mathbb{R}^m$ we have
\begin{align*}
\mathcal{W}_2(X_t(x),\mu)
=\mathcal{W}_2(\mathcal{N}(e^{-At}x,\Sigma_t),\mathcal{N}(0,\Sigma_\infty))
\leq  
 \left(|e^{-At}x|^2+ m\|\Sigma^{1/2}_t-\Sigma^{1/2}_\infty\|^2\right)^{1/2},
\end{align*}
where 
\begin{equation}\label{eq:defsigmanew}
\Sigma_t:=\int_{0}^{t} e^{-A s}\sigma \sigma^* e^{-A^* s} \ud s\quad \textrm{ and }\quad 
\Sigma_\infty:=\int_{0}^{\infty} e^{-A s}\sigma \sigma^* e^{-A^* s} \ud s.
\end{equation}
If, in addition, $\Sigma_t$ commutes with $\Sigma_\infty$ 
it follows that
\begin{equation}\label{e:GaussW2}
\mathcal{W}_2(X_t(x),\mu) 
=\left(|e^{-At}x|^2+ m\|\Sigma^{1/2}_t-\Sigma^{1/2}_\infty\|^2\right)^{1/2}.
\end{equation}
\end{corollary}

\noindent The quadratic variation estimate in Corollary~\ref{cor:GaussW2} can be generalized 
to the L\'evy case. 

\begin{corollary}
Under Hypotheses of Theorem~\ref{thm:Gauss} for some $p>0$ we have
\begin{equation}\label{eq:lowerupperp}
\mathcal{W}_p(X_t(x),\mu)\leq \left(|e^{-At}x|^p+ \mathbb{E}\left[[[\int_{\bullet}^{\infty}e^{-Ar}\sigma \ud {L}_r]]_t^p\right]\right)^{1/p}
\end{equation}
for all $t\geq 0$ and $x\in \mathbb{R}^m$, where $[[\cdot]]_\cdot$ denotes the quadratic variation process \cite{Protter}. In addition, there exists a positive constant $K_p$ such that 
\[
\mathbb{E}\left[[[\int_{\bullet}^{\infty}e^{-Ar}\sigma \ud {L}_r]]_t^p\right]\leq K_p\int_{t}^\infty \int_{|z|<1} |e^{-Ar}\sigma z|^p \nu(\ud z) \ud r.
\]
Note that the right-hand side of the previous inequality needs not to be finite in general.
\end{corollary}

\begin{remark}\label{rem:commnormal}
\begin{enumerate}
 \item We stress that in general the trace in \eqref{eq:tracita} is hard to  compute. 
\item We also point out that the commutativity of $\Sigma_t$ and $\Sigma_\infty$ is hard to verify due to \eqref{eq:defsigmanew}. Inspecting the expression
\begin{equation}
\Sigma_t\Sigma_\infty=\int_{0}^{t} \int_{0}^{\infty} e^{-A s}\sigma \sigma^* e^{-A^* s}  e^{-A r}\sigma \sigma^* e^{-A^* r} \ud s\ud r
\end{equation}
even for $\sigma=I_{m}$
one can see that the  commutativity of $\Sigma_t$ and $\Sigma_\infty$ is equivalent to the normality of $A$, that is, $A^*A=AA^*$. In this case we have 
\begin{equation}
\begin{split}
\Sigma_t\Sigma_\infty&=\int_{0}^{t} \int_{0}^{\infty}  e^{-(A+A^*) r}  e^{-(A+A^*)s}  \ud s\ud r=
\int_{0}^{t} \int_{0}^{\infty}  e^{-(A+A^*) r}  e^{-(A+A^*)s}  \ud s\ud r\\
&=
\int_{0}^{t} e^{-(A+A^*) r} \ud r (A+A^*)^{-1}
=-(A+A^*)^{-1}(e^{-(A+A^*) t}-I_{m}) (A+A^*)^{-1}\\
&=
\int_{0}^{t} e^{-(A+A^*) r} \ud r (A+A^*)^{-1}
=(A+A^*)^{-2}(I_{m}-e^{-(A+A^*) t}).
\end{split}
\end{equation}

\item If $L=B=(B_1,\ldots,B_n)$ is a standard Brownian motion in $\mathbb{R}^n$ 
it follows that
\begin{equation}
\frac{\ud}{\ud t} \Sigma_t=-A\Sigma_t-\Sigma_t A^*+\sigma\sigma^*.
\end{equation}
Assume that $-A$ is Hurwitz stable. Then we have $m^x_t\to 0$ as $t\to \infty$. Moreover, 
$\Sigma_t\to \Sigma_\infty$, where $\Sigma_\infty$ is the unique solution of the matrix Lyapunov equation
\begin{equation}
(-A)\Sigma_\infty+\Sigma_\infty (-A)^*+\sigma\sigma^*=0.
\end{equation}
It has unique solution when $\sigma\sigma^*$ is positive definite.
Note that the precise formula \eqref{eq:formulagauss}
may be hard to compute explicitly, we refer to \cite[Theorem 1, page 443]{LT85} and \cite{Ja68}. 
\end{enumerate}
\end{remark}

\bigskip 
\section{\textbf{Cutoff stability for Hurwitz-stable OU-systems}}\label{s:cutoffstability}

\noindent The main motivation is to first establish the phenomenon with the help of explicit formulas for the Gaussian OU{. In the sequel we} then use the ergodicity bounds {established} in Section~2 to establish the cutoff stability for generic situations {of} L\'evy-OU processes. 

\bigskip 
\subsection{\textbf{Cutoff stability of OU-systems with normal drift and Brownian forcing}}\label{sec:cutsta}\hfill\\

\noindent We apply Theorem~\ref{th:OUGcut} to establish cutoff stability for this process. 

\bigskip
\begin{corollary}[Cutoff stability for $\wW_2$ for non-degenerate Gaussian forcing] \label{cor:OUGcut}\hfill\\
Assume the hypotheses of Theorem~\ref{th:OUGcut} and fix some $x\in \RR^m$. 
\begin{enumerate}
 \item If $x\neq 0$ with $\langle x, v_1\rangle \neq 0$, then we have 
the following cutoff stability for $t_\e := \frac{1}{\mathsf{Re}(\la_1)} |\ln(\e)|$  
\begin{align}\label{e:Gaussiancutoff}
\lim_{\e\ra 0}\frac{\wW_2(X_{\delta \cdot t_\e}(x), \mu)}{\e} 
&=\begin{cases} \infty & \mbox{ for }\delta \in (0,1),\\ 0 & \mbox{ for }\delta >1.\end{cases}
\end{align}
\item If $x\neq 0$ with $\langle x, v_1\rangle =  0$ and   
\[
\rho := \min\{\mathsf{Re}(\lambda_j):j\in\{1,\ldots,m\},
\langle x,v_j \rangle\neq 0
\}<2\mathsf{Re}(\lambda_1), 
\]
then we have the cutoff stability \eqref{e:Gaussiancutoff} for $t_\e := \frac{1}{\rho} |\ln(\e)|$.\\ 
\item If $x = 0\in \RR^m$, we have the cutoff stability \eqref{e:Gaussiancutoff} for $t_\e := \frac{1}{2 \mathsf{Re}(\la_1)} |\ln(\e)|$. 
\end{enumerate}
\end{corollary}
\noindent The proof of Corollary~\ref{cor:OUGcut} is straightforward with the help of the formulas obtained in Theorem~\ref{th:OUGcut}. In fact, Corollary~\ref{cor:OUGcut} can be further sharpened as follows.   

\begin{corollary}[Window cutoff stability]\label{cor:window}
Assume the hypotheses of Theorem~\ref{th:OUGcut} and fix some $x\in \RR^m$. Then we have 
\begin{equation}\label{e: cutoff}
\begin{cases} 
\liminf\limits_{\e\ra 0}~\e^{-1} \cdot \wW_2(X_{t_\e+r}(x), \mu) \ra \infty  &\mbox{ as } r\ra\infty,\\
\limsup\limits_{\e\ra 0}~\e^{-1} \cdot \wW_2(X_{t_\e+r}(x), \mu) \ra  
0 &\mbox{ as }r\ra-\infty.
 \end{cases}
\end{equation}
 \end{corollary}
 
 \bigskip 
{ 
 \begin{remark} 
 We stress, that our results do not need the spectrum of the infinitesimal generator 
 of the Fokker-Planck equation 
 \[
 \frac{\ud u}{\ud t} =  \mathcal{A}^* u,
 \]
 where  
 \[
 \mathcal{A}^* f = -\sum_{i,j=1}^n A_{ij} \frac{\partial}{\partial x_i} (x_j  f) + \frac{1}{2} \sum_{i,j= 1}^n \sigma \sigma^* \frac{\partial^2 f}{\partial x_i \partial x_j}, 
 \]
 which is an infinite-dimensional problem. Instead we only need the spectrum of the matrix $A$. 
 \end{remark}
}

\noindent As mentioned in Remark~\ref{rem:tracita} the case of degenerate noise is hard to treat explicitly, in particular the formulas obtained in Theorem~\ref{th:OUGcut} are not valid. 
However, we present the very special case of a damped 1d harmonic oscillator perturbed by a (non-small) Brownian motion, where this applies but where explicit calculations can still be carried out. 
Nevertheless, it is only in the subsequent section 
that we can establish cutoff stability, for instance for the $m$-dimensional damped harmonic oscillator perturbed by a $m$-dimensional L\'evy process, including a $m$-dimensional Brownian motion. 
 
\begin{example}[Cutoff stability of a harmonic oscillator driven by Brownian motion]\label{ex:oscBM}\hfill\\ 
We consider the harmonic oscillator in 1 dimension with friction $\gamma>0$, perturbed 
by a $1$-dimensional Brownian motion $B$ given by the system of stochastic differential equations 
\begin{align*}
&\ud X_t = Y_t \ud t \\
&\ud Y_t = (-\gamma Y_t - \kappa X_t) \ud t + \varsigma \ud B_t,
\end{align*}
and $(X_0, Y_0)^* = (x, y)^* \in \RR^{2}$, where $\kappa, \varsigma>0$. 
That is, the system $Z_t = (X_t, Y_t)^*$ satisfies 
\[
\ud Z_t = -A Z_t \ud t + \sigma \ud \mathcal{B}_t, \qquad Z_0 = z,
\]
where $\mathcal{B}_t = (0, B_t)^*$, $z = (x, y)^*$, 
\begin{align*}
A = \left(\begin{array}{cc} 0 & -1\\\kappa & \gamma \end{array}\right) \qquad \mbox{ and }\qquad 
\sigma = \varsigma P, \quad \mbox{ where } \quad P := \sigma \left(\begin{array}{cc} 0 & 0\\0 & 1 \end{array}\right).
\end{align*}
Hence any {$t\geq 0$ the marginal} $Z_t$ satisfies $\mathcal{N}(e^{-At}z, \Sigma_t)$. 
We denote the characteristic polynomial given by 
$P(u) = u^2 + \gamma u + \kappa$, 
whose roots are given by 
\[
u_\pm  = -\frac{\gamma}{2} \pm \frac{1}{2} \sqrt{\Delta}, \qquad \mbox{ where } \qquad \Delta = \gamma^2 - 4 \kappa. 
\]
In the sequel we consider the most relevant case of subcritical damping: $\Delta < 0$. Due to $\sigma \sigma^* = \varsigma^2 P$ we calculate with the help of Maple (2022) 
\begin{align*}
\Sigma_t 
&= \int_0^t e^{-As} \sigma  \sigma^*e^{-A^*s} \ud s = \int_0^t e^{-As} \sigma e^{-A^*s} \ud s\\
&= \sigma_{22}^2 \int_0^t  \left(\begin{array}{cc} (e^{-At})_{12}^2& (e^{-At})_{22} \cdot (e^{-At})_{12}  \\[3mm] (e^{-At})_{22} \cdot (e^{-At})_{12} &(e^{-At})_{22}^2
\end{array} \right) \ud s,\\
\end{align*}
where the components formally read 
\begin{align*}
&(\Sigma_t)_{11} = -\frac{{\mathrm e}^{-t \gamma -t \sqrt{\Delta}} \gamma^{2}+{\mathrm e}^{-t \gamma +t \sqrt{\Delta}} \gamma^{2}-{\mathrm e}^{-t \gamma -t \sqrt{\Delta}} \gamma  \sqrt{\Delta}+{\mathrm e}^{-t \gamma +t \sqrt{\Delta}} \gamma  \sqrt{\Delta}-2 \gamma^{2}-8 \,{\mathrm e}^{-t \gamma} \kappa +8 \kappa}{\gamma  \Delta \left(\gamma +\sqrt{\Delta}\right) \left(\gamma -\sqrt{\Delta}\right)},
\\
&(\Sigma_t)_{22} = -\frac{\kappa \left({\mathrm e}^{-t \gamma -t \sqrt{\Delta}} \gamma^{2}+{\mathrm e}^{-t \gamma +t \sqrt{\Delta}} \gamma^{2}+{\mathrm e}^{-t \gamma -t \sqrt{\Delta}} \gamma  \sqrt{\Delta}-{\mathrm e}^{-t \gamma +t \sqrt{\Delta}} \gamma  \sqrt{\Delta}-2 \gamma^{2}-8 \,{\mathrm e}^{-t \gamma} k +8 \kappa \right)}{\Delta \left(\gamma +\sqrt{\Delta}\right) \left(\gamma -\sqrt{\Delta}\right) \gamma},
\\
&(\Sigma_t)_{12} = (\Sigma_t)_{21} = \frac{{\mathrm e}^{-t \gamma -t \sqrt{\Delta}}+{\mathrm e}^{-t \gamma +t \sqrt{\Delta}}-2 \,{\mathrm e}^{-t \gamma}}{2 \Delta }.
\end{align*}
\noindent Since $\Delta<0$ we have that 
\begin{align*}\label{(14)}
&\mathcal{W}_{2}^2(X_t(0), \mu) \\
&= \frac{1}{2 \kappa  \gamma  \Delta }\cdot \bigg(-\sqrt{2}\, \kappa  \Delta \cdot  \sqrt{
\frac{
I_1
}{
\kappa^{2} \Delta 
}
}-\gamma^{2} {\mathrm e}^{-\gamma  t} \left(\kappa +1\right) \cos \! \left(\sqrt{|\Delta|}\, t \right)-\gamma  \,{\mathrm e}^{-\gamma  t} \sqrt{|\Delta|}\, \left(\kappa -1\right) \sin \! \left(\sqrt{|\Delta|}\, t \right)\\
&\qquad +\left(4 \kappa^{2}+4 \kappa \right) {\mathrm e}^{-\gamma  t}- \Delta 
\left(-2+\kappa I_2 \right)\bigg),
\end{align*}
where 
\begin{align*}
I_1 &= \left(\sqrt{I_{10}}
-\gamma^{2} (\kappa^{2}+1) I_{11}+\left(-\gamma  \,\kappa^{2}+\gamma \right) \sqrt{|\Delta|}\, I_{12}+( \Delta {\mathrm e}^{\gamma  t}+4 \kappa ) (\kappa^{2}+1)\right) {\mathrm e}^{-\gamma  t},\\
I_{10} &= -2 \sqrt{|\Delta|}\, (-\gamma^{2} I_{11}+\Delta {\mathrm e}^{\gamma  t}+4 \kappa ) (\kappa +1) (\kappa -1) (\kappa^{2}+1) \gamma  I_{12}\\
&+2 \gamma^{2} (\gamma^{2} \kappa^{4}-2 \kappa^{5}+4 \kappa^{3}
+\gamma^{2}-2 \kappa ) I_{11}^{2}-2 \Big((\kappa -1)^{2} (\kappa +1)^{2} \Delta  {\mathrm e}^{\gamma  t}+4 \kappa  (\kappa^{2}+1)^{2}\Big) \gamma^{2} I_{11}\\
&+{\mathrm e}^{2 \gamma  t}(\kappa -1)^{2} (\kappa +1)^{2} \Delta^{2} +8 \kappa  (\kappa -1)^{2} (\kappa +1)^{2} \Delta +16 \kappa^{6}+4 \kappa^{5} \gamma^{2}\\
&+(-\gamma^{4}-32) \kappa^{4}+24 \kappa^{3} \gamma^{2}+(-2 \gamma^{4}+16) \kappa^{2}+4 \gamma^{2} \kappa -\gamma^{4},\\[2mm]
I_{11} &=  \cos \left(\sqrt{|\Delta|}\, t \right),\qquad 
I_{12} =  \sin \left(\sqrt{|\Delta|}\, t \right),\qquad 
I_2 = \sqrt{-\frac{I_{21}
}{\kappa^{2} \Delta}}-2,\qquad \mbox{ and }
\end{align*}
\begin{align*}
I_{21}&= 2 \,{\mathrm e}^{-\gamma  t} \bigg(\Big\{-2 \sqrt{|\Delta|}\, \left(-\gamma^{2} I_{11}+\Delta {\mathrm e}^{\gamma  t}+4 \kappa \right) (\kappa^4 -1)  \gamma  I_{12}+2 \gamma^{2} \left(\gamma^{2} \kappa^{4}-2 \kappa^{5}+4 \kappa^{3}+\gamma^{2}-2 \kappa \right) I_{11}^{2}\\
&-2 \left((\kappa^2 -1)^{2} \Delta {\mathrm e}^{\gamma  t}+4 \kappa  (\kappa^{2}+1)^{2}\right) \gamma^{2} I_{11} +\left(\kappa -1\right)^{2} \left(\kappa +1\right)^{2} \left(\Delta \right)^{2} {\mathrm e}^{2 \gamma  t}\\
&+8 \kappa  \left(\kappa -1\right)^{2} \left(\kappa +1\right)^{2} \left(\Delta \right) {\mathrm e}^{\gamma  t}+16 \kappa^{6}+4 \kappa^{5} \gamma^{2}+\left(-\gamma^{4}-32\right) \kappa^{4}\\
&+24 \kappa^{3} \gamma^{2}+\left(-2 \gamma^{4}+16\right) \kappa^{2}+4 \gamma^{2} \kappa -\gamma^{4}\Big \}^\frac{1}{2} +\gamma^{2} (\kappa^{2}+1) I_{11}\\
&+\gamma  \sqrt{|\Delta|}\, \left(\kappa -1\right) \left(\kappa +1\right) I_{12}-\Delta {\mathrm e}^{\gamma  t}+4 \kappa  (\kappa^{2}+1)\bigg).\\
\end{align*}
Identifying the highest order exponential term in $\mathcal{W}_{2}^2(X_t(0), \mu)$ 
one can check that  
\begin{equation}\label{e:liminfsup}
0 <\liminf_{t\ra\infty} e^{2\gamma t} \wW_{2}^2(X_t(0), \mu)
\leq \limsup_{t\ra\infty} e^{2\gamma t} \wW_{2}^2(X_t(0), \mu)  < \infty. 
\end{equation}
We illustrate a special case is illustrated in Figure~\ref{fig:plot}.  
\begin{figure}
  \includegraphics[scale=0.5]{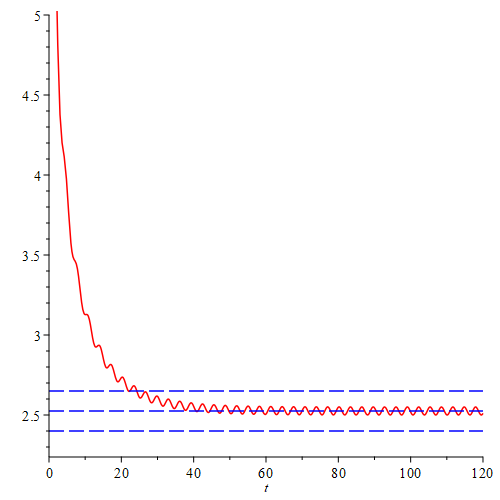}\\
  \caption{Plot of $t \mapsto e^{2\gamma t}\cdot \wW_{2}^2(X_t(0), \mu)$ for $\kappa = 1$, $\gamma = 10^{-1}$ (red curve). 
  The upper and the lower limit differ clearly due to the oscillations above and below the value 2.55 (central dashed blue curve).}\label{fig:plot}
\end{figure}
As a bottom line, we have verified the asymptotics of Theorem~\ref{thm:cutoffWp} of order $e^{-2\gamma t}$ by direct calculation for the degenerate case of the harmonic oscillator with moderate Brownian forcing.  Similarly to the case of the small noise regime as treated in \cite[Section 4.2.4]{BHPWA}, subcritical damping does not exhibit a true limit in \eqref{e:liminfsup}, as clearly seen by the oscillations in Figure~\ref{fig:plot}. 
\end{example}

\bigskip  
\subsection{\textbf{Cutoff stability of generic OU-systems driven with L\'evy forcing}}\label{sec:32}\hfill\\

\noindent In this subsection we treat general $\sigma \in \RR^{m\times n}$, $L$ with values in $\RR^n$ with finite first moment and $A \in \RR^{m\times m}$ Hurwitz stable. Additionally, we assume that $A$ has the following generic structure.  

\noindent 
\begin{definition}[Generic interaction force]\label{def:generic}
We say that $A\in \RR^{m\times m}$ is generic, if it has $m$ different (possibly complex valued) eigenvalues $\lambda_1, \dots, \lambda_m$. 
\end{definition}
\noindent In this case we have for $x\in \RR^m$ 
\[
e^{-At}x= \sum_{j=1}^m e^{-\lambda_j t}c_j(x) v_j, \quad \mbox{ where } c_j(x)\in \CC, 
\]
and $\{v_1, \dots, v_m\}$ is a basis of eigenvectors of $\CC^m$ with unit length. Note that the eigenvectors are not necessarily orthogonal. One of the main consequences of genericity in the preceding sense is the following. 
\begin{lemma}\label{lem:expgenerica}
Let $A$ satisfy Hypothesis~\ref{hyp:stable} and be generic in the sense of Definition~\ref{def:generic}. 
Then for each $x\in \mathbb{R}^m$, $x\neq 0$,  
there exist $\rho = \rho(x) >0$ and $C_i(x)>0$, $i=1,2,$ such that 
\begin{equation}\label{eq:expgenerica}
C_1(x)e^{-\rho t}\leq |e^{-At}x|\leq C_2(x)e^{-\rho t}\quad \textrm{ for all }\quad t\geq 0.
\end{equation}
\end{lemma}
\noindent The proof is given in Appendix~\ref{a:expgenerica}. With this result in mind we now state the main theorem. 

\begin{theorem}[Generic cutoff stability for L\'evy Ornstein-Uhlenbeck systems]\label{thm:cutoffWp}\hfill\\
Let $A$ satisfy Hypothesis~\ref{hyp:stable} and assume that $A$ is generic in the sense of Definition~\ref{def:generic}. We assume that $L$ satisfies Hypothesis~\ref{hyp:moments} for some $p\geq 1$. In addition $\sigma$ satisfies Hypothesis~\ref{hyp:diff}. 
For $x\neq 0$ such that $x \neq A^{-1}\sigma \mathbb{E}[L_1]$ choose $\rho_x>0$ as in \eqref{eq:expgenerica} and set 
\begin{align*}
t_\e = \frac{1}{\rho_x} |\ln(\e)|.
\end{align*}
Then for all $1 \leq q \leq p$ we have 
\begin{align}\label{eq:limitWgen}
\lim_{\e\ra 0}\frac{\wW_q(X_{\delta \cdot t_\e}(x), \mu)}{\e} = \left\{
\begin{array}{cl} 
\infty, & \mbox{for }\delta\in (0, 1),\\
0, & \mbox{for }\delta >1.
\end{array}
 \right.\\\nonumber
\end{align}
\end{theorem}
\noindent Theorem~\ref{thm:cutoffWp} generalizes Corollary~\ref{cor:OUGcut} for any given initial condition $x$ to the case of a generic matrix $A$ and non-Gaussian L\'evy noise with first moments. In addition, it covers degenerate noise. For instance, Example~\ref{ex:oscBM} is covered without any of the lengthy calculations. In Example~\ref{ex:oscJC} we show how even more complex systems such as coupled chains of oscillators with moderate external heat bath is included. The proof is given after the subsequent corollary. 

\noindent Since convergence in the WKR-distance of order $p\geq 1$ is equivalent to the simultaneous convergence in distribution and the convergence of the absolute moments of order $p\geq 1$, see \cite[Theorem 6.9]{Vi09}, we also obtain the respective (pre-)cutoff stability for the $p$-th absolute moments.

\begin{corollary}[Observable pre-cutoff stability]\label{cor:observable}
Assume the hypotheses and notation of Theorem~\ref{thm:cutoffWp}. 
Then for all $1 \leq q \leq p$ and $x\neq 0$ we have for all $\delta>1$  
\begin{align*}
\lim_{\e\ra 0}~\frac{1}{\e}\cdot\Big|\EE\big[|X_{\delta \cdot t_\e}(x)|^q\big] - \EE\big[|X_\infty|^q\big]\Big|
&= 0. 
\end{align*}
\end{corollary}

\medskip 
\begin{proof}[Proof of Theorem~\ref{thm:cutoffWp}: ] 
Let $x\in \mathbb{R}^m$ and $t\geq 0$.  
By Item~(1) of Theorem~\ref{thm:couplings} and \eqref{eq:expgenerica} we have
\begin{equation}
\begin{split}
\frac{\mathcal{W}_p(X_t(x),\mu)}{\e}
\leq \frac{|e^{-At}x|}{\e}+ \frac{\mathcal{W}_p(X_t(0),\mu)}{\e}\leq \frac{|e^{-At}x|}{\e}+\frac{1}{\e}\int_{\mathbb{R}^m} |e^{-At}z|\mu(\ud z)\leq C(x)\frac{1}{\e}e^{-\lambda_* t}.
\end{split}
\end{equation}
Note that $e^{-\lambda_* t_\e}/\e=1$ and $t_{\e}\to \infty$ as $\e \to 0$. Hence, \eqref{eq:limitWgen} is valid for $\delta>1$. 

In the sequel, we show that \eqref{eq:limitWgen} is valid for $0<\delta<1$. By Lemma~\ref{thm:couplings} we have 
\begin{equation}
\left|\frac{e^{-At}x}{\e}+\frac{\mathbb{E}[X_t(0)]}{\e}-\frac{\int_{\mathbb{R}^m} z \mu(\ud z)}{\e} \right|\leq 
\frac{\mathcal{W}_p(X_t(x),\mu)}{\e}
\end{equation}
for all $t\geq 0$ and $x\in \mathbb{R}^m$, where
\begin{align*}
\mathbb{E}[X_t(0)]&=e^{-At} \int_{0}^{t} e^{As}\sigma \mathbb{E}[L_1] \ud s =\int_{0}^{t}e^{-A(t-s)}\ud s \sigma \mathbb{E}[L_1]\\
&=\int_{0}^{t}e^{-As}\ud s \sigma \mathbb{E}[L_1]=A^{-1}(I_m-e^{-At})\sigma \mathbb{E}[L_1].
\end{align*}
Hence sending $t\ra\infty$ we have in abuse of notation for $X_\infty \stackrel{d}{=} \mu$ that 
\[
\mathbb{E}[X_\infty]=A^{-1}\sigma \mathbb{E}[L_1],
\]
which gives
\[
\mathbb{E}[X_t(0)]-\mathbb{E}[X_\infty]=-A^{-1}e^{-At}\sigma \mathbb{E}[L_1].
\]
Then by \eqref{eq:expgenerica} 
\begin{align*}
\left|\frac{e^{-At}x}{\e}+\frac{\mathbb{E}[X_t(0)]}{\e}-\frac{\int_{\mathbb{R}^m} z \mu(\ud z)}{\e} \right|
&= \left|\frac{e^{-At}x}{\e}+\frac{\mathbb{E}[X_t(0)]}{\e}-\frac{\mathbb{E}[X_\infty]}{\e} \right|\\
&=\left|\frac{e^{-At}(x-A^{-1}\sigma \mathbb{E}[L_1])}{\e} \right|\\
&\geq C_2(x-A^{-1}\sigma \mathbb{E}[L_1]) 
C_1(x)\frac{e^{-\lambda_* t}}{\e}.
\end{align*}
Again, since $e^{-\lambda_* t_\e}/\e=1$ and $t_{\e}\to \infty$ as $\e \to 0$, \eqref{eq:limitWgen} is valid for $\delta\in (0,1)$. This finishes the proof. 
\end{proof}

\noindent In the sequel, we show Corollary~\ref{cor:observable} for which we use the following lemma, shown in \cite[p.972, Lemma B.2]{BL23}. 

\begin{lemma}\label{lem:elementary}
For any $p\geq 1$ we have the estimate 
\begin{align*}
| |x|^p - |y|^p |  \leq p \Big(|x|^{p-1} + |y|^{p-1}\Big) |x-y|, \qquad x, y\in \RR^m. 
\end{align*}
\end{lemma}

\medskip

\begin{proof}[Proof of Corollary~\ref{cor:observable}: ] With the help of Lemma~\ref{lem:elementary} for $p\geq 1$ and H\"older's inequality 
we have for any coupling $\pi$ between $X_t$ and $X_\infty$ that 
\begin{align*}
|\EE[|X_t|^p] - \EE[|X_\infty|^p]| 
&\leq p \EE_\pi [(|X_t|^{p-1} + |X_\infty|^{p-1}) |X_t-X_\infty|]\\ 
&\leq p \EE_\pi[(|X_t|^{p-1} + |X_\infty|^{p-1})^\frac{p}{p-1}]^\frac{p-1}{p} 
\EE_\pi[|X_t-X_\infty|^p]^\frac{1}{p}\\
&\leq p 2^{\frac{p}{p-1}-1}\Big(\EE[|X_t|^{p}]^\frac{p-1}{p} 
+ \EE[|X_\infty|^{p})]^\frac{p-1}{p}\Big)
\EE_\pi[|X_t-X_\infty|^p]^\frac{1}{p}. 
\end{align*}
Bearing in mind that 
\[
\EE[|X_t|^{p-1}]^\frac{1}{p-1} \leq |e^{-At} x| + \EE[|X_t(0)|^{p-1}]^\frac{1}{p-1} 
\]
we have 
\begin{align*}
|\EE[|X_t|^p] - \EE[|X_\infty|^p]| 
&\leq p 2^{\frac{2}{p-1}}\Big(|e^{-At} x|^p + \EE[|X_t(0)|^{p-1}]^\frac{p}{p-1}´
+ \EE[|X_\infty|^{p-1})]^\frac{p-1}{p}\Big)
\EE_\pi[|X_t-X_\infty|^p]^\frac{1}{p}. 
\end{align*}
Optimizing over all such couplings $\pi$ we obtain 
\begin{align*}
|\EE[|X_t|^p] - \EE[|X_\infty|^p]| 
&\leq p 2^{\frac{2}{p-1}}\Big(|e^{-At} x|^p + \EE[|X_t(0)|^{p-1}]^\frac{p}{p-1}
+ \EE[|X_\infty|^{p-1})]^\frac{p-1}{p}\Big)
\wW_{p}(X_t, X_\infty), 
\end{align*}
such that 
\begin{align*}
\frac{1}{\e} |\EE[|X_t|^p] - \EE[|X_\infty|^p]| 
&\leq p 2^{\frac{2}{p-1}}\Big(|e^{-At} x|^p + 2\EE[|X_\infty|^{p-1})]^\frac{p-1}{p}\Big)
\frac{\wW_{p}(X_t, X_\infty)}{\e}. 
\end{align*}
By Theorem~\ref{thm:cutoffWp} the desired result follows for $\delta>1$. 
\end{proof}

\bigskip 
\noindent In fact, the result can be further sharpened (without proof), as follows. 

\begin{theorem}[Window cutoff stability]\label{thm:window}
Assume the hypotheses and notation of Theorem~\ref{thm:cutoffWp}. 
Then for all $1 \leq q \leq p$ and $x\neq 0$ 
 such that $x \neq A^{-1}\sigma \mathbb{E}[L_1]$ 
we have 
\begin{equation}\label{e: cutoffwindows}
\lim_{r\ra\infty} \begin{cases} 
\liminf\limits_{\e\ra 0}~\e^{-1} \cdot \wW_q(X_{t_\e-r}(x), \mu) =  
\infty, &\\ 
\limsup\limits_{\e\ra 0}~\e^{-1} \cdot \wW_q(X_{t_\e+r}(x), \mu) = 0.  &\\
 \end{cases}
\end{equation}
 \end{theorem}
 
 \bigskip
 {
\section{Examples} 
We stress that in this section the matrices $A$ that appears in the examples below are generic in the sense of Definition~\ref{def:generic}, and the quantitative upper-lower bounds given in Theorem~\ref{thm:couplings} are valid and available with less effort than lengthy computations, which we illustrate below for specific models. Moreover, our quantitative upper-lower bounds cover the situation of a multidimensional undecoupled L\'evy noise with finite first moment
and the $\mathcal{W}_p$ for  any $p\geq 1$.
By Theorem~\ref{thm:cutoffWp} we obtain cutoff stability at explicitly given time scale $t_\e$.
\begin{example}
The simplest example of our setting is obviously the ordinary scalar (L\'evy-)Ornstein-Uhlenbeck process given by 
\[
\ud X_t(x) = -\lambda X_t(t) \ud t + \sigma \ud L_t, \qquad X_0(x) = x\in \RR, 
\]
where $\lambda, \sigma>0, x\in \RR$ and $L = (L_t)_{t\geq 0}$ is a L\'evy process (including Brownian motion) with finite first moment. 
Denote by $\mu$ the unique limiting measure or dynamical equilibrium. 
We refer to Remark~\ref{rem:explan} Item~(3).
Then 
Theorem~\ref{thm:cutoffWp} implies 
for $t_\e:= -\frac{1}{\lambda} \ln(\e)$, $\e>0$, that 
\[
\lim_{\e \ra 0}\frac{\wW_p(X_{\delta t_\e}(x), \mu)}{\e} = \begin{cases} \infty & \delta \in (0,1),\\ 0 &\delta >1.\end{cases} 
\]
In other words, for $\delta>1$ we have 
\[
\wW_p(X_{\delta t_\e}(x), \mu) \ll \e,\quad \mbox{ as }\quad \e \ra 0,
\]
and for $\delta<1$
\[
\wW_p(X_{\delta t_\e}(x), \mu) \gg \e,\quad \mbox{ as }\quad \e \ra 0.
\]
\end{example}
}

\bigskip 
{
\begin{example}\label{ex:gyrator}
The second simplest class of examples consists of the so-called Brownian gyrator (see \cite{DBT23}) given by the solution of the following SDE  
\begin{equation}
\ud 
X_t(x) = 
\left(
\begin{array}{cc} 
-(\gamma+\kappa) & \kappa\\
 \kappa & -(\gamma+\kappa)
\end{array}
\right) X_t(x) \ud t
+ 
\left(
\begin{array}{cc} 
\sigma_1 & 0 \\
0 & \sigma_2
\end{array}
\right) \ud B_t, \qquad X_0(x) = x, 
\end{equation}
where $X_t(x) = (X_{t, 1}(x), X_{t, 2}(x))$, $B_t = (B_{t,1}, B_{t, 2})$, and $x = (x_1, x_2)$.  
It represents the positions $X_{t,1}$ and $X_{t,2}$ of two Brownian particles with unit mass which evolve, with common friction constant $\gamma>0$ and mutual spring interaction with spring constant $\kappa>0$. Each of the particles is connected with an individual heat bath, which are generically at different temperature $\sigma_1\not=\sigma_2$. 
Such systems have been studied in theoretical contexts, such as models of two temperature diffusions, models of minimal heat engines on the nanoscale, keystone examples for the control theory of the harmonic oscillator, and recently also implemented by optical experiments, see \cite[p.3, left column]{BPS20}. 
In this context, we may not apply Theorem~\ref{thm:Gauss} due to $\sigma_1\neq \sigma_2$. 
However, still we may calculate directly formula \eqref{e:Pythagoras}
\begin{align*}
\wW_2(X_t(x), \mu)&=
\left(
|e^{-At }\widetilde{x}|^2+ \textsf{Tr}\left({\Sigma_t+\Sigma_\infty-2(\Sigma^{1/2}_\infty \Sigma_t \Sigma^{1/2}_\infty)^{1/2}}\right)\right)^{1/2}\\
\end{align*}
where 
\[
A = \left(
\begin{array}{cc} 
-(\gamma+\kappa) & \kappa\\
 \kappa & -(\gamma+\kappa)
\end{array}
\right) \qquad \mbox{ and }\qquad 
\sigma = \mathsf{diag}(\sigma_1, \sigma_2). 
\]
Straightforward computations yields that 
\[
e^{At} =  
\left(\begin{array}{cc}
\frac{{\mathrm e}^{-t \gamma}}{2}+\frac{{\mathrm e}^{-t \gamma -2 t \kappa}}{2} & -\frac{{\mathrm e}^{-t \gamma -2 t \kappa}}{2}+\frac{{\mathrm e}^{-t \gamma}}{2} 
\\
 -\frac{{\mathrm e}^{-t \gamma -2 t \kappa}}{2}+\frac{{\mathrm e}^{-t \gamma}}{2} & \frac{{\mathrm e}^{-t \gamma}}{2}+\frac{{\mathrm e}^{-t \gamma -2 t \kappa}}{2} 
\end{array}\right).
\]
For $\sigma=\sigma_1=\sigma_2>0$ we have
\[
\begin{split}
\mathcal{W}_2(X_t(0),\mu)&=2\sigma^2\frac{\gamma+\kappa}{(\gamma+2\kappa)\gamma}-
\frac{\sigma^{2}}{2(\gamma+2\kappa)}{\mathrm e}^{-2 t \left(\gamma +2 \kappa \right)}  -\frac{\sigma^{2}}{2\gamma} {\mathrm e}^{-2 \gamma  t}\\
&\qquad
-\frac{2 \gamma\sigma^2  \sqrt{ \left(1-e^{-(2\gamma+4\kappa)t}\right)}-2\left(\gamma +2 \kappa \right)\sigma^2 \sqrt{1-e^{-2\gamma t}}}{2 \left(\gamma +2 \kappa \right) \gamma}
\end{split}
\]
yielding that 
\[
\lim\limits_{t\to\infty}\mathcal{W}_2(X_t(0),\mu)=2\sigma^2\frac{\gamma+\kappa}{(\gamma+2\kappa)\gamma}-\frac{\sigma^2}{(\gamma+2\kappa)}-\frac{\sigma^2}{\gamma}=0.
\]
 The reason that this limit exists in comparison to Example~\ref{ex:oscBM} is the lack of imaginary parts in the spectrum of the interaction matrix due to its symmetry, and hence diagonalizability. Such a profile cannot be guaranteed in general. In \cite{BHPWA}, the authors characterized the existence of a cutoff profile in terms of orthogonality properties of the generalized eigenvectors of the interaction matrix. 
With the help of Maple (2023) 
we calculate explicitly 
\begin{align*}
\wW_{2}(X_t(0), \mu)= \frac{J_1(t)}{I_4} &-\frac{
\sqrt{2}\, I_2}{I_4} 
\sqrt{
\frac{
((I_7
-4\, \sqrt{
\frac{
(J_{10}(t) + J_9(t)
) 
I_5^{2} I_1^{2}
}{
J_5(t)
}
}) ({\mathrm e}^{\gamma  t})^{2}
-I_1^{2} I_2^{2} I_5^{2}) 
({\mathrm e}^{t \kappa})^{4}
+J_3(t)
}{
J_2(t)
}
}\\
&-\frac{\sqrt{2}\, I_2}{I_4} \sqrt{
\frac{
(
(
I_{7}+4\, \sqrt{
\frac{
(J_{10}(t) -J_9(t))
I_5^{2} I_1^{2}
}{
J_5(t)
}
}
) 
({\mathrm e}^{\gamma  t})^{2}
-I_1^{2} I_2^{2} I_5^{2}) ({\mathrm e}^{t \kappa})^{4}
+J_3(t)
}{
J_2(t)
}
},
\end{align*}
where 
\begin{align*}
I_1 &=  \gamma +2 \kappa,\\
I_2 &= \sigma_1^{2}+\sigma_2^{2},\\
I_3 &= \sigma_1^{4}+\sigma_2^{4},\\
I_4 &= 4 I_1 \gamma,\\
I_5 &= \gamma +\kappa,\\
I_6 &= \gamma^{2} I_1^{2} I_5^{2},\\
I_7 &= I_3 4 \gamma^{4}+16 \kappa  I_3 \gamma^{3}
+22 \kappa^{2} (\sigma_{1}^{4}+\frac{6}{11} \sigma_{1}^{2} \sigma_{2}^{2}+\sigma_{2}^{4}) \gamma^{2}
+12 \kappa^{3} I_1^{2} \gamma 
+4 \kappa^{4} I_1^{2},\\
I_8 &= ((\gamma^{3}+3 \gamma^{2} \kappa +2 \gamma  \,\kappa^{2}-\kappa^{3}) \sigma_{1}^{4}-2 \sigma_{2}^{2} (\gamma^{3}+5 \gamma^{2} \kappa +6 \gamma  \,\kappa^{2}+\kappa^{3}) \sigma_{1}^{2}+\sigma_{2}^{4} (\gamma^{3}+3 \gamma^{2} \kappa +2 \gamma  \,\kappa^{2}-\kappa^{3})) \gamma, \\
I_9 &= (\gamma^{3}+3 \gamma^{2} \kappa +2 \gamma  \,\kappa^{2}+\kappa^{3})  \sigma_{1}^{4}
-2 \sigma_{2}^{2} (\gamma^{3}
+\gamma^{2} \kappa 
-2 \gamma  \,\kappa^{2}
-\kappa^{3}) \sigma_{1}^{2},\\
J_1(t) &= (4 \gamma +4 \kappa -{\mathrm e}^{-2 t I_2} \gamma -I_2 {\mathrm e}^{-2 \gamma  t} ) I_1,\\
J_2(t) &= ({\mathrm e}^{t \kappa})^{4} ({\mathrm e}^{\gamma  t})^{2} I_2^{2} I_5^{2},\\
J_3(t) &=-2 \gamma^{2} (\sigma_{1}-\sigma_{2})^{2} (\sigma_{1}
+\sigma_{2})^{2} I_2^{2} ({\mathrm e}^{t \kappa})^{2}
-I_6,\\
J_5(t) &= ({\mathrm e}^{\gamma  t})^{4} ({\mathrm e}^{t \kappa})^{8},\\
J_6(t) &= {(I_2 ({\mathrm e}^{t \kappa})^{2}+\gamma )}^{2} (I_1^{2} I_2^{2} I_5^{2} ({\mathrm e}^{t \kappa})^{4}
+2 ((\gamma^{2}+2 \gamma  \kappa -\kappa^{2}) \sigma_{1}^{4}
-6 (\gamma^{2}+2 \gamma  \kappa +\frac{1}{3} \kappa^{2}) \sigma_{2}^{2} \sigma_{1}^{2}\\
&\quad+\sigma_{2}^{4} (\gamma^{2}+2 \gamma  \kappa -\kappa^{2})) I_2 \gamma  ({\mathrm e}^{t \kappa})^{2}+I_6),\\
J_8(t) &= (
((\gamma^{3}+3 \gamma^{2} \kappa +2 \gamma  \,\kappa^{2}+\kappa^{3}) \sigma_{1}^{4}
-2 \sigma_{2}^{2} (\gamma^{3}+\gamma^{2} \kappa -2 \gamma  \,\kappa^{2}-\kappa^{3}) \sigma_{1}^{2}\\
&\quad+\sigma_{2}^{4} (\gamma^{3}+3 \gamma^{2} \kappa +2 \gamma  \,\kappa^{2}+\kappa^{3})) I_2 ({\mathrm e}^{t \kappa})^{2}
+
I_8
),\\ 
J_9(t) &= -\frac{1}{2}
I_5 (I_2 ({\mathrm e}^{t \kappa})^{2}+\gamma )  J_8(t)
({\mathrm e}^{t \kappa})^{4} ({\mathrm e}^{\gamma  t})^{2}
+\frac{
J_6(t)
}{
16
},\\
J_{10}(t) &= I_5^{2} ((\gamma^{4}+4 \gamma^{3} \kappa +5 \gamma^{2} \kappa^{2}+2 \gamma  \,\kappa^{3}+\kappa^{4}) \sigma_{1}^{4}-2 \sigma_{2}^{2} (\gamma^{2}+3 \gamma  \kappa +\kappa^{2}) (\gamma^{2}+\gamma  \kappa -\kappa^{2}) \sigma_{1}^{2} \\
&\quad+\sigma_{2}^{4} (\gamma^{4}+4 \gamma^{3} \kappa +5 \gamma^{2} \kappa^{2}+2 \gamma  \,\kappa^{3}+\kappa^{4})) 
({\mathrm e}^{t \kappa})^{8} ({\mathrm e}^{\gamma  t})^{4}.
\end{align*}
In Figure~\ref{fig:plot2} we observe the non-oscillatory asymptotic decay of order $e^{-2t}$ of $\wW_{2}(X_t(0), \mu)$ for the parameters
$\kappa = 0.8$, $\gamma = 1$, $\sigma_1=2$ and $\sigma_2 = 1$.
\begin{figure}
  \includegraphics[scale=0.4]{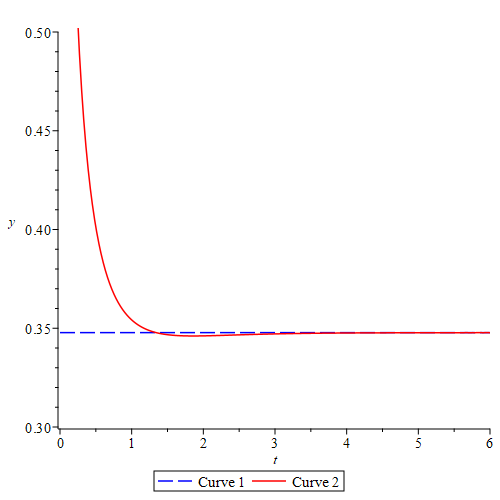}\\
  \caption{Plot of $e^{2 t} \wW_{2}(X_t(0), \mu)$ for $\kappa = 0.8$, $\gamma = 1$, $\sigma_1=2$ and $\sigma_2 = 1$ (red curve). The precise limiting value is given by $\frac{3205}{9216}$ (blue dashed line).}\label{fig:plot2}
\end{figure} 
\end{example}
}

{
\begin{example}[A biophysical 
transcription-translation model in equilibrium]\label{ex:bio}
The next more complex system which falls under our scenario is the case of a stochastic linear harmonic oscillator. It is the basic equilibrium scenario in many applications in the sciences, for instance in a bio-physical transcription-translation model of mRNA and proteins concentrations, \cite[p.1251, left column, first display]{CPLO07} for constant DNA-mRNA transcription rate $k_B>0$ and constant internal transcriptional noise level $q$. The positive constants $\gamma_R$ and $\gamma_p$ represent the rate of degradation of the mRNA and the protein, while $k_p>0$ is the represents the necessary amount of mRNA needed in order to produce a protein. 
\begin{align*}
\frac{\ud x}{\ud t} &= k_R - \gamma_R x + q\dot B,\\
\frac{\ud y}{\ud t} &= k_p x - \gamma_p y,  
\end{align*}
which reads as follows: 
\begin{equation}
\ud 
\left(
\begin{array}{c} 
x\\y
\end{array}
\right)
= 
\left(
\begin{array}{c} 
k_B\\
0
\end{array}
\right) \ud t
+
\left(
\begin{array}{cc} 
-\gamma_R & 0\\
k_p & -\gamma_p
\end{array}
\right)
\left(
\begin{array}{c} 
x\\
y
\end{array}
\right) \ud t
+ 
\left(
\begin{array}{c} 
1\\
0
\end{array}
\right) \ud B
\end{equation}
with invariant distribution 
\[
\mathcal{N}\left(\frac{k_R}{\gamma_R}, \frac{q^2}{2 \gamma_R}\right) \otimes \mathcal{N}\left(\frac{k_p k_R}{\gamma_p \gamma_R}, \frac{q^2 k_p^2}{2\gamma_R \gamma^2_p}\right).
\]
We note that neither the ergodicity bound of Theorem~\ref{thm:Gauss}, 
nor the existence of a cutoff profile $\mathcal{P}_{(x_1, x_)}(r)$ 
as in Example~\ref{ex:gyrator} is valid, as carried out in Example~\ref{ex:oscBM}. 
Instead, Theorem~\ref{thm:cutoffWp}, Corollary \ref{cor:OUGcut} and Corollary \ref{cor:window} yields 
for generic initial values $x= (x_1, x_2)\neq 0$, $\langle x, v_1\rangle \neq 0$, then we have 
for $t_\e := \frac{1}{\mathsf{Re}(\la_1)} |\ln(\e)|$ the following cutoff stability 
\begin{equation}
\begin{cases} 
\liminf\limits_{\e\ra 0}~\e^{-1} \cdot \wW_2(X_{t_\e+r}(x), \mu) \ra \infty  &\mbox{ as } r\ra\infty,\\
\limsup\limits_{\e\ra 0}~\e^{-1} \cdot \wW_2(X_{t_\e+r}(x), \mu) \ra  
0 &\mbox{ as }r\ra-\infty.
 \end{cases}
\end{equation}
where $\gamma = \min\{\gamma_R, \gamma_p\}>0$. 
Similar calculations for $\mathcal{W}_2(X_t(0),\mu)$ 
as in Example~\ref{ex:gyrator}
are carried out in Example~\ref{ex:oscBM}. 
\end{example}
}

 \bigskip 
\begin{example}[Cutoff stability of a Jacobi chain under fixed amplitude L\'evy forcing with first moments]\label{ex:oscJC}\hfill\\
This example comes from the recent work  
\cite[Section 4.1]{RAQUEPAS} and \cite[Section 4.2]{JPS17}. 
In \cite[Subsections 3.2.1 and 3.2.2]{BHPWA} 
the cutoff thermalization for small noise was discussed thoroughly. We show that Theorem~\ref{thm:cutoffWp} covers cutoff stability in the sense of \eqref{e:cutoffstability} even for heat baths a fixed temperature in this benchmark example.  
Consider the following Hamiltonian of $m$ coupled scalar oscillators
\begin{align*}
H:\RR^m & \times  \RR^m \ra \RR,\qquad (p, q) \mapsto H(p, q) := \frac{1}{2} \sum_{i=1}^m p_i^2 + \frac{1}{2} \sum_{i=1}^{m} \gamma q_i^2  
+ \frac{1}{2} \sum_{i=1}^{m-1} \kappa (q_{i+1}- q_i)^2.
\end{align*}
Coupling the first and the $m$-th oscillator to a Langevin heat bath each 
with positive temperatures $\varsigma_1^2$ and $\varsigma_m^2$,  
a coupling constant $\kappa>0$ and friction coefficient $\gamma>0$ yields 
for $X = (X^{1}, \dots, X^{2m}) = (p, q) = (p_1, \dots, p_m, q_1, \dots, q_m)$ 
the $2m$-dimensional system 
\begin{align}\label{e:JCsystem}
\ud X^\e_t = -A X^\e_t \ud t + \sigma \ud L_t,
\end{align}
where $A$ is a $2m$-dimensional real square matrix of the following shape 
\begin{align*}
A = \left(
\begin{array}{ccccc|ccccccc}
\varsigma_1 &0& \dots & &0&\kappa + \gamma &-\kappa&0&& \dots  & 0\\             
  0& 0 &\dots& &&-\kappa & 2\kappa + \gamma & - \kappa& 0& \dots &\\             
&\vdots &\ddots&&\vdots&0&-\kappa &2\kappa + \gamma&-\kappa&0&\vdots \\
\vdots &&\ddots&\vdots &&&0&\ddots&\ddots&\ddots&0\\
 &&\dots&0 &0 & \vdots &&0 &-\kappa & 2\kappa + \gamma & - \kappa\\             
0&& \dots & 0& \varsigma_n& 0& \dots &&0&-\kappa & \kappa + \gamma\\
\hline
-1&0& &\dots&0& 0 &&\dots&  & &0\\             
0&-1& 0&\dots&& &&&  &&\\             
\vdots&0&\ddots &\ddots&\vdots&\vdots& &&  &&\vdots \\             
&\ddots&0 &-1&0& &&& &&\\             
0&&\dots &0&-1& 0 &&\dots&  &&0\\             
\end{array}
\right),  
\end{align*}
and $L_t = (L^1_t, 0, \dots, L^m_t, 0, \dots, 0)^*$. 
Here $L^1, L^m$ are one dimensional independent L\'evy processes satisfying Hypothesis \ref{hyp:moments} for some $p\geq 1$. By Section 4.1 in \cite{RAQUEPAS} $A$ satisfies Hypothesis \ref{hyp:stable}. 
Consequently, if $A$ - in addition - is generic in the sense of Definition~\ref{def:generic},  
Theorem \ref{thm:cutoffWp} implies that the system exhibits cutoff stability towards its unique invariant measure $\mu$ in the sense of \eqref{e:cutoffstability} for some time scale $t_\e$.  
We illustrate this phenomenon for the following choice of parameters $\varsigma_1=\varsigma_m=\kappa=1$, $\gamma=0.01$ and $m=5$ with the help of Wolfram Mathematica 12.1. 
In this particular case the interaction matrix $A$ looks as follows 
\[\left( \begin {array}{cccccccccc} 1&0&0&0&0&1.01&-1&0&0&0\\ \noalign{\medskip}0&0&0&0&0&-1&2.01&-1&0&0\\ \noalign{\medskip}0&0&0&0&0&0&-1&2.01&-1&0\\ \noalign{\medskip}0&0&0&0&0&0&0&-1&2.01&-1\\ \noalign{\medskip}0&0&0&0&1&0&0&0&-1&1.01\\ \noalign{\medskip}-1&0&0&0&0&0&0&0&0&0\\ \noalign{\medskip}0&-1&0&0&0&0&0&0&0&0\\ \noalign{\medskip}0&0&-1&0&0&0&0&0&0&0\\ \noalign{\medskip}0&0&0&-1&0&0&0&0&0&0\\ \noalign{\medskip}0&0&0&0&-1&0&0&0&0&0\end {array} \right) \]
with the following vector of eigenvalues 
\[
\left(
\begin{array}{l}
\la_1\\
\bar \la_1\\
\la_2\\
\bar \la_2\\
\la_3\\
\bar \la_3\\
\la_4\\
\bar \la_4\\
\la_5\\
\la_6
\end{array}
\right)
=
\left(
\begin{array}{l}
0.0263377 + 1.88656\cdot i\\ 0.0263377 - 1.88656\cdot i\\
0.104782 + 1.55549\cdot i\\
0.104782 - 1.55549\cdot i\\
0.234099 + 1.06262\cdot i\\ 
0.234099 - 1.06262\cdot i\\ 
0.395218 + 0.517319\cdot i\\ 
0.395218 - 0.517319\cdot i\\ 
0.452655 + 0.\cdot i\\ 
0.0264706 + 0.\cdot i
\end{array}
\right)\qquad\mbox{ with } \rho = \mathsf{Re}(\la_1), \theta = \mathsf{Im}(\la_1). 
\]
Since all eigenvalues are different $A$ is generic in the sense of Definition~\ref{def:generic}. 
Therefore the solution $X$ of the system \eqref{e:JCsystem} satisfies the hypotheses of Theorem~\ref{thm:cutoffWp} for all initial values $x\neq 0$ such that $x \neq A^{-1} \sigma \EE[L_1] = A^{-1}(\varsigma_1, 0, \dots, 0, \varsigma_m, 0, \dots, 0)^*$. 
\end{example}

{
\begin{example}[More general networks]\label{ex:oscmoregeneral}\hfill\\
\begin{enumerate}
 \item For more general network topologies of harmonic oscillators 
with some of the oscillators connected to heat reservoirs at different temperatures 
we refer to the works of \cite{BLL04, CEHR18, JPS17, RAQUEPAS}. 
While the authors there typically work with non-linear interaction potential, 
our situation only covers the case of quadratic potentials. 
In \cite{BLL04} the authors study crystal type extensions of linear Jacobi chains, which were generalized in 
\cite{CEHR18, JPS17, RAQUEPAS}. 
\item The admissible network topologies in \cite{JPS17, RAQUEPAS} between heat reservoirs and the spring interaction of the springs are hidden in terms of the controllability of $-A$ and $\sigma$, which is equivalent to the well-known Kalman condition of the existence of some $m_*\leq m$ such that 
\[
\mathsf{span}\{\sigma e_i, A \sigma e_i, A^2 \sigma e_i, \dots, \sigma A^{m_*-1} e_i, i=1, \dots, m\} = \mathbb{R}^m. 
\]
\item In \cite{CEHR18} the authors give an explicit construction for sufficient conditions on the controllability in terms of the network topology, which turns the graph of connected springs via a linear sequence of ''nicely connected`` layers of spring masses. Given a finite set of masses $\mathcal{G}$ and the connections $E \subset \mathcal{G} \times \mathcal{G}$. 
Consider the set $\mathcal{B}\subset \mathcal{G}$ connected to the heat reservoirs. 
 Then $\mathcal{B}$ is \textit{nicely connected} to a vertex $v\in \mathcal{G}\setminus \mathcal{B}$ ($\mathcal{B} \leadsto v$, for short) 
 if there exists $b\in \mathcal{B}$ such that 
$(b, v)\in E$, but $b$ is not connected to any other vertes $v'\in \mathcal{G}\setminus \mathcal{B}$. It is worth noting, that for $\mathcal{B} \leadsto v$ it is necessary that at least one $b\in \mathcal{B}$ satisfies the preceding condition, while all other connections of $v$ to $b' \in \mathcal{B}$ might violate it. 
If we denote by $\mathcal{T}\mathcal{B}$ (the first layer of) all vertices $v\in \mathcal{G}\setminus \mathcal{B}$ to which $\mathcal{B}$ is ''nicely connected`` to, and if $\mathcal{G} = \bigcup_{n\geq n_0} \mathcal{T}^n \mathcal{B}$, where $\mathcal{T}^{n+1} \mathcal{B} = \mathcal{T} (\mathcal{T}^{n} \mathcal{B})$, $n\geq 0$, then condition C1 in \cite{CEHR18} is satisfied. Under additional conditions C2-C5, that is, non-degeneracy of the (possibly nonlinear) interaction potentials (C2), homogeneity and coercivity of the (possibly nonlinear) interaction potentials (C3), the local injectivity of the interaction forces (C4) and the asymptotic domination of the interaction potentials over the pinning potentials (C5), there is an exponential convergence of the convergence in law. Natural applications for these kind of systems are for instance the micromolecular dynamics of the dendritic spine of a neuronal cell, see \cite[Chapter 5, Subsection 5.2.9]{Sch13} formula (5.27). 
\item We present a simple network of three completely connected oscillators with one heat reservoir 
connected to the first mass, see Figure~\ref{fig:graph}, which does not satisfy (C1) in \cite{CEHR18}: 
\begin{figure}
\begin{tikzpicture}[->,>=stealth',shorten >=2pt, line width=1pt, 
                                  node distance=2cm, style ={minimum size=10mm}]
\tikzstyle{every node}=[font=\large]
\node [circle, draw] (a) at (0,0) {1};
\node                (d) at (-3,0) {};
 \path  (a) edge [loop above] node {$-3$} (a);
\node [circle, draw] (b)  at (4,0) {2};
 \path  (b) edge [loop right] node {$-3$} (b);
\node [circle, draw] (c) at (2,-3) {3};
\draw[->] (a) to[bend left] node[above] {$1$} (b);
\draw[->] (b) to node[above] {$1$}(a);
 \path  (c) edge[loop below] node {$-3$} (c);
\draw[->] (a) to node[above] {$2$} (c);
\draw[->] (c) to[bend left] node[above] {$2$}(a);
\draw[->] (b) to[bend left] node[above] {$1$}(c);
\draw[->] (c) to node[above] {$1$} (b);
\draw[->] (d) to node[above] {$\sigma_1 = 1$} (a);
\end{tikzpicture}
\caption{Transition graph for a 3 component connected oscillator with a heat reservoir in the first mass.}\label{fig:graph}
\end{figure}
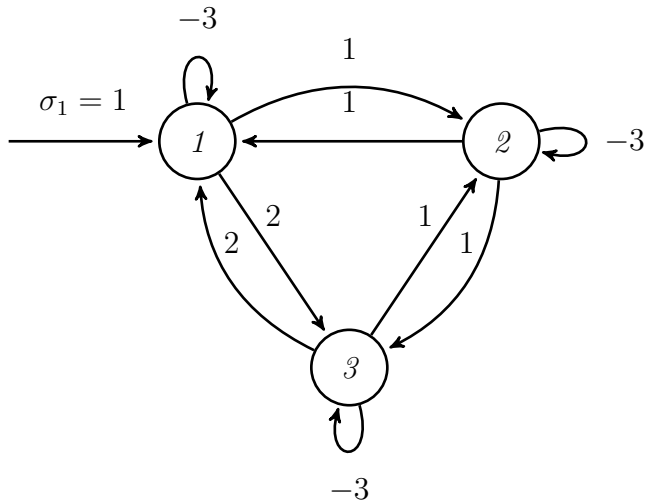
The respective stochastic differential equation satisfies 
\[
\ud X_t(x) = - AX_t(x) \ud t + \sigma \ud B, \qquad X_0(x) = x,  
\]
where 
\[
-A = 
\left(
\begin{array}{cccccc}  
1  & 0 &0 & 3 & -1  & -2\\
0  & 0 &0 & -1 & 3  & -1  \\
0  & 0 &0 & -2 & -1 & 3\\
-1 & 0 &0 & 0  & 0  & 0 \\
0 & -1 &0 & 0  & 0  & 0 \\
0 & 0 & -1 & 0  & 0  & 0 \\
\end{array}
\right)
\]
It is clear by definition of ''nicely connectedness`` that the node $1$ does not control the complete graph. 
However, the real parts of the spectrum $\{\lambda_1, \bar \lambda_1, \lambda_2, \bar \lambda_2, \lambda_3, \bar \lambda_3\}$ are strictly negative 
\[
\lambda_1 \approx  - 0.25039 + 2.12688 i,\qquad \lambda_3 \approx  - 0.04139 + 1.96062 i, 
\qquad \lambda_5 \approx - 0.208261 + 0.49001 i, 
\]
such that $-A$ is Hurwitz stable and generic in the sense of Definition~\ref{def:generic}. 
After the lengthy but explicit calculations for the Brownian gyrator and the oscillator in Example~\ref{ex:oscBM}, it is obvious, that symbolic calculations could still be carried out, but become increasingly infeasible. 
\item Note that even if we generalize $B = L$ being a scalar L\'evy process, the (suboptimal) ergodicity (upper and lower) bounds of Theorem~\ref{thm:couplings} and the Gaussian (upper and lower) bounds in Theorem~\ref{thm:Gauss} remain valid and yield an exponential convergence towards the invariant measure at a rate which is proportional to $e^{-0.04139 t}$. 
\item In addition, Theorem~\ref{thm:cutoffWp} and Theorem~\ref{thm:cutoffWp} yield (simple) cutoff stability and window cutoff stability in the sense of item (1) and (2) in the introduction, for generic initial values $x$ along the asymptotic time scale $t_\e := -\ln(\e)/0.04139 $, $\e\in (0,1)$. Corollary~\ref{cor:observable} implies precutoff for all existing higher absolute moments of the $X$ along the same time scale $t_\e$. 
\item The preceding result highlights the advantage of the WKR-distance, since for our results in Section~\ref{s:ergodicitybounds} and Section~\ref{s:cutoffstability} we need not satisfy any of controllability (or irreducibility) properties, in contrast to typical for the total variation or the relative entropy. 
\end{enumerate}
\end{example}
}

\bigskip

{
\section{\textbf{Conclusion }} 
This article provides upper and lower bounds on the WKR-p distance between the time $t$ marginal of a multidimensional Ornstein-Uhlenbeck process with fixed (non-small) (Brownian or L\'evy) noise  amplitude and their respective dynamic equilibria, see Theorem~\ref{thm:couplings}. 
We also establish a new identity for WKR between Ornstein-Uhlenbeck systems driven by non-degenerate Brownian motion $\sigma\sigma^*=I$  with normal (or diagonalizable) interaction matrix, see Theorem~\ref{th:OUGcut}.
Such identity shows the following thermalization scenario as time $t$ grows: 
fast adaptation of the scale at the scale of the limiting distribution followed by a subsequent recentering of the location at a slower pace.  
This type of behavior is conjectured to be true for more general L\'evy driven systems. \\
These non-asymptotic results are applied for cutoff stability, that is, abrupt thermalization 
to $\e$ small distances in WKR along a particular $\e$-dependent time scale in Theorem~\ref{thm:cutoffWp} and Theorem~\ref{thm:window}. In Corollary~\ref{cor:observable} it is shown that the observables in our general setting also converge abruptly to the moments of the limiting distribution.\\ 
Applications are the Brownian or L\'evy gyrator, a single harmonic oscillator for instance in a genetic transcription-translation model, Jacobi-chains of linear oscillators with a heat bath in the extremes and more general network topologies.  For the single harmonic oscillator and the Brownian gyrator the WKR-2 distances are calculated explicitly illustrating the limitations of explicit formulas.
}

\section*{\textbf{Acknowledgments}}
\noindent
G.B. would like to express his gratitude to University of Helsinki, Department of Mathematics and
Statistics, for all the facilities used along the realization of this work. 
The authors thank Prof. Juan Manuel Pedraza, Physics Department at Universidad de los Andes, for helpful discussions, which have led to Example~\ref{ex:bio} and Example~\ref{ex:oscmoregeneral}.
They also thank the anonymous referees for the careful reading and helpful suggestions which have improved the quality of the manuscript.

\medskip

\noindent
\textbf{Funding.}
\hfill

\noindent
The research of G.B. has been supported by the Academy of Finland,
via an Academy project (project No. 339228) and the Finnish Centre of Excellence in Randomness and STructures (project No. 346306).
The research of M.A.H. has been supported by 
the project ``Mean deviation frequencies and the cutoff phenomenon'' (INV-2023-162-2850) 
of the School of Sciences (Facultad de Ciencias) at Universidad de los Andes.

\noindent

\medskip

\noindent
\textbf{Ethical approval.}
\hfill

\noindent
Not applicable.

\medskip

\noindent
\textbf{Competing interests.}
\hfill

\noindent
The authors declare that they have no conflict of interest.

\medskip

\noindent
\textbf{Authors' contributions.}
\hfill

\noindent
All authors have contributed equally to the paper.

\medskip

\noindent
\textbf{Availability of data and materials.}
\hfill

\noindent
Data sharing not applicable to this article as no data-sets were generated or analyzed during the current study.

\bigskip 

\appendix
\section{\textbf{Properties of the WKR-distance}}\label{a:WKRproperties}

\noindent Recall the WKR distance $\wW_p$ of order $p$ given in Definition~\ref{def:Wasserstein}. 

\begin{lemma}[Properties of the WKR distance]\label{lem:properties}\hfill\\
Let $p>0$, $x,y\in \RR^m$ be deterministic vectors, $c\in \RR$ and $X, Y$ be random vectors in $\RR^m$ with finite $p$-th moment. Then we have: 
\begin{itemize}
\item[a)] The WKR distance is a \textbf{metric} (or distance), 
in the sense of being definite, symmetric and satisfying the triangle inequality.
\item[b)] \textbf{Translation invariance:} 
$\mathcal{W}_p(x+X,y+Y)=\mathcal{W}_p(x-y+X,Y)$.
\item[c)] \textbf{Homogeneity:} 
\[
\mathcal{W}_p(cX,cY)=
\begin{cases}
|c|\;\mathcal{W}_p(X,Y),&\textrm{ if } p\in [1,\infty),\\
|c|^{p}\;\mathcal{W}_p(X,Y),&\textrm{if } p\in (0,1).
\end{cases}
\]
\item[d)] \textbf{Shift linearity:} 
For $p\geq 1$ it follows 
\begin{equation}\label{eq:shitflinearity}
\mathcal{W}_p(x+X,X)=|x|.
\end{equation}
For $p\in(0,1)$ equality (\ref{eq:shitflinearity}) is false in general. However,  {it holds} the following inequality
\begin{equation}\label{ec:cotaabajop01}
\max\{|x|^{p}-2\EE[|X|^p],0\}\leq 
\mathcal{W}_p(x+X,X)\leq |x|^{p}.
\end{equation}
\item[e)] \textbf{Domination:} For any given coupling $\mathcal{T}$ between 
$X$ and $Y$ it follows  
\[
\mathcal{W}_p(X, Y) \leq \Big(\int_{\mathbb{R}^m\times \mathbb{R}^m} |u-v|^p \mathcal{T}(\ud u,\ud v)\Big)^{\min\{1/p,1\}}.
\]
\item[f)] \textbf{Characterization:} Let $(X_n)_{n\in \mathbb{N}}$ be a sequence of random vectors with finite $p$-th moments 
and $X$ a random vector with finite $p$-th moment{. Then} the following {statements} are equivalent: 
\begin{enumerate}
 \item $\mathcal{W}_p(X_n, X) \rightarrow 0$ as $n\rightarrow \infty$. 
 \item $X_n \stackrel{d}{\longrightarrow} X$ as $n \rightarrow \infty$ and $\EE[|X_n|^p] \rightarrow \EE[|X|^p]$ as $n\rightarrow \infty$. \\
\end{enumerate}
\item[g)] \textbf{Contractivity:} Let $F:\mathbb{R}^m \to \mathbb{R}^k$, $k\in \mathbb{N}$, be Lipschitz continuous with Lipschitz constant~$1$. Then for any $p>0$ 
\begin{equation}
\mathcal{W}_p(F(X),F(Y))\leq \mathcal{W}_p(X,Y).
\end{equation}
\hfill
\end{itemize}
\end{lemma}

For the proof we refer to \cite{Vi09} and for Item d) and g) to \cite[Lemma 2.2]{BHPWA}.

\bigskip

 \section{\textbf{Proof of Theorem~\ref{thm:couplings}}}\label{A:proofcouplings}

\begin{proof}[Proof of Theorem~\ref{thm:couplings}] 
We start with the proof of the first estimates in Item (1) and Item (2) $p\geq 1$. 
Note that for any random variable $X$ with finite $p$-th moment, and any deterministic vector $u\in \mathbb{R}^m$ it follows the recently established so-called shift linearity property, see \cite[Lemma~2.2~(d)]{BHPWA}
\begin{equation}\label{eq:shiftl}
\Wp(u+X,X)=|u|.
\end{equation}
Recall the decomposition~\eqref{eq:repre}.
The preceding equality with the help of the triangle inequality for $\Wp$ implies
\begin{equation}
\begin{split}
\mathcal{W}_p(X_t(x),\mu) &\leq \mathcal{W}_p(X_t(x),X_t(0))+\mathcal{W}_p(X_t(0),\mu)\\
& =\mathcal{W}_p(e^{-At}x+X_t(0),X_t(0))+\mathcal{W}_p(X_t(0),\mu)\\
&=|e^{-At}x|+\mathcal{W}_p(X_t(0),\mu).
\end{split}
\end{equation}
Conversely, the shift linearity~\eqref{eq:shiftl} and the triangle inequality for $\Wp$ yield
\begin{equation}
\begin{split}
|e^{-At}x|-\mathcal{W}_p(X_t(0),\mu) = \mathcal{W}_p(X_t(x),X_t(0))-\mathcal{W}_p(X_t(0),\mu)\leq \mathcal{W}_p(X_t(x),\mu).
\end{split}
\end{equation}
This shows the first inequalities in Item (1) and Item (2).

We continue with the proof of the second estimate of Item (2). 
We start with the observation that
\begin{equation}\label{eq:infcoupling}
\left|\mathbb{E}[X]-\mathbb{E}[Y]\right|\leq \mathcal{W}_1(X,Y)
\end{equation}
for any $X$ and $Y$ random {vectors} with finite first moment. Indeed, 
\begin{equation}
\left|\mathbb{E}[X]-\mathbb{E}[Y]\right|=
\left|
\int_{\mathbb{R}^m \times \mathbb{R}^m} {u}\Pi(\ud u,\ud v)-\int_{\mathbb{R}^m \times \mathbb{R}^m} v\Pi(\ud u,\ud v)\right|\leq 
\int_{\mathbb{R}^m} |u-v|\Pi(\ud u,\ud v)
\end{equation}
for any coupling $\Pi$ between $X$ and $Y$.
Minimizing over all couplings $\Pi$ we deduce the second inequality in Item (2) ($p\geq 1$).
Since 
\[
\mathbb{E}[X_t(x)]=e^{-At}x+e^{-At} \int_{0}^{t} e^{As}\sigma \mathbb{E}[L_1] \ud s,
\]
the proof of inequality 
the second inequality in Item (2) now follows straightforwardly with the help of Jensen's inequality.

In the sequel, we show the third inequality in Item (2).
For $p\in (0,1)$ we have 
\[\max\{|u|^p-2\mathbb{E}[|X|^p],0\}\leq \Wp(u+X,X)\leq |u|^p,\]
see \cite[Lemma 2.2~(d)]{BHPWA} and similarly we obtain the third statement in Item (2). 

\noindent Now, we prove the second inequality in Item~(1).
Using the same noise (synchronous coupling) in \eqref{eq:modelo0} we have
\[
\ud (X_t(x)-X_t(y))=-A(X_t(x)-X_t(y))\ud t,
\]
which yields pathwise
$X_t(x)-X_t(y)=e^{-At}(x-y)$ for all $t\geq 0$, $x,y \in \mathbb{R}^m$.
By the definition of $\Wp$ we have
\begin{equation}\label{eq:xy}
\Wp(X_t(x),X_t(y))\leq  |e^{-At}(x-y)|^{\min\{1,p\}}. 
\end{equation}
Finally, since the solution process~\eqref{eq:modelo0} is Markovian,
the disintegration inequality for $\Wp$ with the help of \eqref{eq:xy} and $X_t(\mu)\stackrel{d}{=}\mu$ yields
\begin{equation}
\begin{split}
\mathcal{W}_p(X_t(x),\mu)=\mathcal{W}_p(X_t(x),X_t(\mu))&\leq 
\int_{\mathbb{R}^m} \mathcal{W}_p(X_t(x),X_t(y))\mu(\ud y)\\
&
\leq \int_{\mathbb{R}^m} |e^{-At}(x-y)|^{\min\{1,p\}}\mu(\ud y).
\end{split}
\end{equation}
The lower bound $0$ in Item (2) is trivial. 
This finishes the proof. 
\end{proof}

\section{\textbf{Proof of Lemma~\ref{lem:Hurwitz}}}\label{a:Hurwitz}

\begin{proof} Item (1) follows directly by {negativity of the real parts of the spectrum.}
We continue with the proof of Item (2) and start with the counterexample. The matrices 
\[
A=\begin{pmatrix}
1 & 1\\
0 & 1
\end{pmatrix}\quad \textrm{ and }\quad 
B=\begin{pmatrix}
1 & 0\\
9 & 1
\end{pmatrix}
\]
are each Hurwitz stable, however, the matrix $A+B$ is not Hurwitz stable since its eigenvalues are  $5$ and $-1$. We now show that the commutativity of $A$ and $B$ implies the Hurwitz stability of $A+B$. 
Let $A, B\in \mathbb{R}^{m\times m}$ be each Hurwitz stable and define 
\[\beta:=\max\limits_{\lambda\in \textsf{Spect(B)}} \textsf{Re}(\lambda).\] Then for any $\epsilon>0$ there exists a positive constant $C_{B,\epsilon}$ such that $\|e^{Bt}\|\leq C_{B,\epsilon}e^{(\beta+\epsilon)t}$ for all $t\geq 0$.

\noindent Let $A$ be a Hurwitz stable square matrix and assume that  $AB=BA$.
Since $A$ is  a Hurwitz stable matrix, then for any $\epsilon>0$ we have the existence of a positive constant $C_{-A,\epsilon}$ such that
\[\|e^{-A t}\|\leq C_{-A,\epsilon}e^{(-\alpha+\epsilon)t}\quad \textrm{ for all }\quad t\geq 0,\] where 
$-\alpha:=\max\limits_{\lambda\in \textsf{Spect(-A)}} \textsf{Re}(\lambda)$. Note that 
$\alpha=\min\limits_{\lambda\in \textsf{Spect(A)}} \textsf{Re}(\lambda)>0$.
If we assume that 
\[
-\alpha+\beta<0,\]
then $A+B$ is Hurwitz stable. Indeed,
the matrix $A+B$ is Hurwitz stable if and only if  the linear system $\dot X_t=-(A+B)X_t$ is asymptotically stable. Note that $X_t=e^{-(A+B)t}x$ for any initial condition $x\in \mathbb{R}^m$.
Since $A$ and $B$ commute, by the Baker-Campbell-Hausdorff-Dynkin formula \cite[Chapter 5]{Ha15} we have  $X_t=e^{-At}e^{-Bt}x$. Then for any $\epsilon>0$, {the }submultiplicativity of the norm implies
\[|X_t|\leq \|e^{-At}\||e^{-Bt}x|\leq C_{-A,\epsilon} e^{(-\alpha+\epsilon) t}C_{B,\epsilon}e^{(\beta+\epsilon) t}|x|=C_{-A,\epsilon}C_{B,\epsilon}e^{(-\alpha+\beta+2\epsilon)t}|x| \quad \textrm{ for all }\quad t\geq 0.
\]
For $\epsilon>0$ small enough, we have that $-\alpha+\beta+2\epsilon<0$. Consequently the linear system $\dot X_t=-(A+B)X_t$ is asymptotically stable. This proves the claim of Item (2).
\end{proof}

\section{\textbf{Proof of Lemma~\ref{lem:expgenerica}}}\label{a:expgenerica}

\begin{proof} We show \eqref{eq:expgenerica} by constructing an appropriate value of $\rho$. 
Denote by  
\[
J_x := \{j\in \{1, \dots, m\}~|~c_j(x) \neq 0\}\neq \emptyset.  
\]
Then 
\[
e^{-At}x= \sum_{j\in J_x} e^{-\lambda_j t}c_j(x) v_j.  
\]
Define $\rho_x := \min\{\textsf{Re}(\lambda_j):j\in J_x\}>0$ 
and $R_x := \{j\in J_x~:~\mathsf{Re}(\lambda_j) = \rho_x\}$. 
Then for $\rho_x$ we have 
\begin{align*}
e^{\rho_x t} e^{-At}x
&=  \sum_{j\in R_x} e^{\rho_x t}e^{-\lambda_jt}c_j(x) v_j +
\sum_{j\in J_x\setminus R_x} e^{\rho_x t} e^{-\lambda_jt} c_j(x) v_j.  
\end{align*}
It is easy to see that the second sum on the right-hand side of the preceding display 
tends to $0$ as $t\ra\infty$. Therefore we obtain 
\[
\limsup\limits_{t\ra\infty}
|e^{\rho_x t} e^{-At}x|=
\limsup\limits_{t\ra\infty} 
|\sum_{j\in R_x} e^{\rho_x t}e^{-\lambda_jt}c_j(x) v_j|.
\]
The $\limsup $ in the preceding expression can be replaced by $\liminf $.
We note that for $\theta_j = \mathsf{Im}(\la_j)$ 
\[
e^{\rho_x t}e^{-\lambda_jt} = e^{-i\theta_j t} \qquad \mbox{ for any }j\in R_x. 
\]
By the triangular inequality we have 
\[
|\sum_{j\in R_x} e^{\rho_x t}e^{-\lambda_jt}c_j(x) v_j| \leq \sum_{j\in R_x} |c_j(x)| < \infty. 
\]
Now, we show that the lower limit is positive. By contradiction{, let us} assume that 
\[
\liminf_{t\ra\infty } |\sum_{j\in R_x} e^{-\theta_jt}c_j(x) v_j| = 0. 
\]
That is, there exists a sequence $(t_k)_{k\in \NN}$ with $t_k\ra\infty$, $k\ra\infty$, 
such that 
\[
\lim_{k\ra\infty } |\sum_{j\in R_x} e^{-\theta_jt_k}c_j(x) v_j| = 0. 
\]
By a diagonal argument we may assume that $\lim_{k\ra\infty} e^{-\theta_jt_k}= z_j$ for all $j\in R_x$, where $|z_j| = 1$. This yields 
\[
\sum_{j\in R_x} c_j(x) z_j v_j =  0, 
\]
which contradicts linear independence of $\{v_1, \dots, v_m\}$ due to the generic choice of $A$. 
In summary, {we have}
\[
0 < \liminf_{t\ra\infty} |e^{\rho_x t} e^{-At}x|\leq 
\limsup_{t\ra\infty} |e^{\rho_x t} e^{-At}x| \leq \sum_{j\in R_x} |c_j(x)| < \infty. 
\]
To deduce \eqref{eq:expgenerica} {it is sufficient} that $t \mapsto |e^{\rho_x t} e^{-At}x|$ is continuous and positive. 
\end{proof}


\begin{thebibliography}{40}

\bibitem{Al83}
Aldous, D., 
``Random walks on finite groups and rapidly mixing Markov chains'',  
{\it Seminar on Probability, XVII. Lecture Notes in Math.}
{\bf 986}, 243-297. Springer, Berlin, 1983.

\bibitem{AD87}
Aldous, D., and Diaconis, P.,  
``Strong uniform times and finite random walks'',  
{\it Adv. in Appl. Math.} {\bf 8}, no. 1 (1987) 69-97.

\bibitem{AD} 
Aldous, D., and Diaconis, P.,  
``Shuffling cards and stopping times'',  
{\it Amer. Math. Monthly} {\bf 93}, no. 5 (1986) 333-348.

\bibitem{APPLEBAUMBOOK} 
Applebaum, D.,  
``L\'evy processes and stochastic calculus'', 
Cambridge University Press, Cambridge, 2004.

{
\bibitem{BPS20}
Baldassarri, A., Puglisi, A., and Sesta, L.,
``Engineered swift equilibration of a Brownian gyrator''
{\it Phys. Rev. E} 102, 030105(R) (2020)
}

\bibitem{Ba18}
Barrera, G., 
``Abrupt convergence for a family of Ornstein-Uhlenbeck processes'', 
\textit{Braz. J. Probab. Stat.} 32 (2018), no. 1, 188--199.

\bibitem{BHPGBM}
Barrera, G., H\"ogele, M.A., and Pardo, J.C.,  
``Non-commutative geometric Brownian motion exhibits nonlinear cutoff stability'',  
{\it https://arxiv.org/abs/2207.01666} 

\bibitem{BHPWA}
Barrera, G., H\"ogele, M.A., and Pardo, J.C.,  
``Cutoff thermalization for Ornstein--Uhlenbeck systems with small L\'evy noise in the Wasserstein distance'', 
{\it J. Stat. Phys.} {\bf 184}, no. 27, (2021).

\bibitem{BHPWANO}
Barrera, G., H\"ogele, M.A., and Pardo, J.C.,  
``The cutoff phenomenon in Wasserstein distance for nonlinear stable Langevin systems with small L\'evy noise'', 
J. Dyn. Diff. Equat. (2022). https://doi.org/10.1007/s10884-022-10138-1

\bibitem{BHPTV}
Barrera, G., H\"ogele, M.A., and Pardo, J.C.,  
``The cutoff phenomenon in total variation for nonlinear Langevin systems with small layered stable noise'',  
{\it Electron. J. Probab.} {\bf 26}, no. 119 (2021) 1-76.

\bibitem{BHPSPDE}
Barrera, G., H\"ogele, M.A., and Pardo, J.C.,  
``The cutoff phenomenon for the stochastic heat and the wave equation subject to small L{\'e}vy noise'', 
\textit{Stoch. Partial Differ. Equ. Anal. Comput.} (2022)

\bibitem{BHPPSHELL}
Barrera, G., H\"ogele, M.A., Pardo, J.C., and Pavlyukevich, I., 
``Cutoff ergodicity bounds in Wasserstein distance for a viscous energy shell model with L{\'e}vy noise''  
{\it https://arxiv.org/abs/2302.13968} 

\bibitem{BJ}
Barrera, G., and Jara, M.,  
``Abrupt convergence of stochastic small perturbations of one dimensional dynamical systems'',  
\textit{J. Stat. Phys.} {\bf 163},  no. 1, (2016) 113-138.

\bibitem{BJ1}
Barrera, G., and Jara, M.,  
``Thermalisation for small random perturbation of hyperbolic dynamical systems'',  
\textit{Ann. Appl. Probab.} {\bf 30},  no. 3 (2020) 1164-1208.

\bibitem{BL21}
Barrera, G., Liu, S., 
``A switch convergence for a small perturbation of a linear recurrence equation'',  
\textit{Braz. J. Probab. Stat.} 35 (2021), no. 2, 224--241.

\bibitem{BL23}
Barrera, G., and Lukkarinen, J.,
``Quantitative control of Wasserstein distance between Brownian motion and the Goldstein-Kac telegraph process'',  
\textit{Ann. Inst. Henri Poincar{\'e} Probab.} 
Stat. 59 (2023), no. 2, 933--982.

\bibitem{BP}
Barrera, G., and Pardo, J.C.,  
``Cut-off phenomenon for Ornstein-Uhlenbeck processes driven by L\'evy processes'', 
\textit{Electron. J. Probab.} {\bf 25}, no. 15 (2020) 1-33, .

\bibitem{BY1} 
Barrera, J, Bertoncini, O., and Fern\'andez, R.,  
``Abrupt convergence and escape behavior for birth and death chains'',   
{\it J. Stat. Phys.} {\bf 137}, no. 4, (2009) 595-623.

\bibitem{BLY06}
Barrera, J., Lachaud, B., and Ycart, B.,  
``Cut-off for $n$-tuples of exponentially converging processes'',  
{\it Stochastic Process. Appl.} {\bf 116}, no. 10 (2006) 1433-1446.

\bibitem{BY2014} 
Barrera, J., and Ycart, B.,  
``Bounds for left and right window cutoffs'',  
{\it ALEA Lat. Am. J. Probab. Math. Stat.} {\bf 11} (2014) 445-458.

{
\bibitem{Ba14}
Barucca, P., 
``Localization in covariance matrices of coupled heterogeneous Ornstein-Uhlenbeck processes'', 
{\it Phys. Rev. E} 90, (2014) 062129. 
}

\bibitem{BASU} 
Basu, R, Hermon, J., and Peres, Y., 
``Characterization of cutoff for reversible Markov chains'',   
{\it Ann. Probab.} {\bf 45} (3), 2017, 1448--1487.

\bibitem{BKR17}
B{\'a}tkai, A., Kramar Fijav\v{z}, M., Rhandi, A., 
``Positive operator semigroups. From finite to infinite dimensions.'' 
With a foreword by Rainer Nagel and Ulf Schlotterbeck. 
\textit{Operator Theory: Advances and Applications,} 257. Birkh\"auser/Springer, Cham, 2017.

\bibitem{BOK10} 
Bayati, B., Owahi, H., and Koumoutsakos, P., 
``A cutoff phenomenon in accelerated stochastic simulations of chemical kinetics
via flow averaging (FLAVOR-SSA)'', 
{\it J. Chem. Phys.} {\bf 133}, 244-117, (2010).

\bibitem{BD92}
Bayer, D., and Diaconis, P.,  
``Trailing the dovetail shuffle to its lair'',  
{\it Ann. Appl. Probab.} {\bf 2}, no. 2 (1992) 294-313.

\bibitem{B-HLP19}
Ben-Hamou, A., Lubetzky, E., and Peres, Y.,  
``Comparing mixing times on sparse random graphs'',  
{\it Ann. Inst. Henri Poincar\'e Probab. Stat.} 
{\bf 55}, no. 2 (2019) 1116-1130.

\bibitem{BBF08}
Bertoncini, O., Barrera, J., and Fern{\'a}ndez, R.,  
``Cut-off and exit from metastability: two sides of the same coin'',  
{\it C. R. Math. Acad. Sci. Paris} {\bf 346}, no. 11-12 (2008) 691-696.

\bibitem{BJL19}
Bhatia, R., Jain, T., Lim, Y., 
''Inequalities for the Wasserstein mean of positive definite matrices``, 
\textit{Linear Algebra Appl.} 576 (2019), 108--123.

\bibitem{BJL19b}
Bhatia, R., Jain, T., Lim, Y., 
''On the Bures-Wasserstein distance between positive definite matrices``, 
\textit{Expo. Math.} 37 (2019), no. 2, 165--191.

\bibitem{Bh07}
Bhatia, R., 
''Positive definite matrices``, 
\textit{Princeton Series in Applied Mathematics.} 
Princeton University Press, Princeton, NJ, 2007.

{
\bibitem{BLL04}
Bonetto, F.,  Lebowitz, J. L., and Lukkarinen J., 
''Fourier’s Law for a Harmonic Crystal with Self-Consistent Stochastic Reservoirs``, 
\textit{Journal of Statistical Physics}, Vol. 116, Nos. 1/4, August 2004
}
\bibitem{BCS19}
Bordenave, C. , Caputo, P., and Salez, J., 
``Cutoff at the ``entropic time'' for sparse Markov chains``, 
{\it Probab. Theory Related Fields} 
{\bf 173}, no. 1-2 (2019) 261-292.

\bibitem{BCS18}
Bordenave, C. , Caputo, P., and Salez, J., 
''Random walk on sparse random digraphs.``  
{\it Probab. Theory Related Fields}, 
{\bf 170}, no. 3-4 (2018) 933-960.

\bibitem{BrzezniakZabczyk}
Brze\'zniak, Z., and Zabczyk, J.,  
''Regularity of Ornstein-Uhlenbeck processes driven by a L\'evy white noise``,  
{\it Potential Anal.} {\bf 32}, no. 2 (2010) 153-188.

{
\bibitem{CPLO07}
Chabot, J., Pedraza, J., Luitel, P. et al. 
''Stochastic gene expression out-of-steady-state in the cyanobacterial circadian clock.`` 
{\it Nature} 450, 1249--1252 (2007). 
}

\bibitem{CSC08} 
Chen, G., and Saloff-Coste, L.,  
''The cutoff phenomenon for ergodic Markov processes``, 
{\it Electron. J. Probab.} {\bf 13}, no. 3 (2008) 26-78.

\bibitem{CT23} 
Chhachhi, S., and Teng, F., 
''On the 1-Wasserstein Distance between Location-Scale Distributions and the Effect of Differential Privacy`` {\it arXiv:2304.14869v1 [math.PR] 28 Apr 2023}.

{
\bibitem{CPP20}
Chigarev, V., Kazakov, A., and Pikovsky, A., 
''Kantorovich-Rubinstein-Wasserstein distance between overlapping attractor and repeller``
{\it Chaos} 30, (2020) 073114. 
}

\bibitem{CS21}
Chleboun, P., and Smith, A.,  
''Cutoff for the square plaquette model on a critical length scale``, 
 {\it Ann. Appl. Probab.} {\bf 31}, no. 2 (2021) 668-702.

{
\bibitem{CEHR18}
Cuneo, N., Eckmann, J.P., Hairer, M., Rey-Bellet, L., 
''Non-equilibrium steady states for networks of oscillators.`` 
Electron. J. Probab. 23 1 - 28, 2018. 
}


\bibitem{CT16}
Czechowski, Z., Telesca, L., 
''Detrended fluctuation analysis of the Ornstein-Uhlenbeck process: Stationarity versus nonstationarity`` 
\textit{Chaos} 26, 113109 (2016)


\bibitem{DIA96} 
Diaconis, P., 
''The cut-off phenomenon in finite Markov chains``,  
{\it Proc. Nat. Acad. Sci. U.S.A.} {\bf 93}, no. 4 (1996) 1659-1664.

\bibitem{DLVV21}
Dinh, T. H., Le, C. T., Vo, B. K., Vuong, T. D., 
''The $\alpha$-z-Bures Wasserstein divergence``, 
Linear Algebra Appl. 624 (2021), 267--280.

\bibitem{DL82}
Dowson, D. C., Landau, B. V., 
''The Fr{\'e}chet distance between multivariate normal distributions``, 
\textit{J. Multivariate Anal.} 12 (1982), no. 3, 450--455.

{
\bibitem{DBT23}
Du Buisson, J., and Touchette, H., 
''Dynamical large deviations of linear diffusions``
\textit{Phys. Rev. E} 107, (2023) 054111.  
}


{
\bibitem{FG21}
Figalli, A., Glaudo, F., 
''An invitation to optimal transport, Wasserstein distances, and gradient flows.`` 
\textit{EMS Textbk. Math.}
EMS Press, Berlin, 2021. 
}

{
\bibitem{GHKM17}
Gairing, J.M., H\"ogele, M.A., Kosenkova, T., and Monahan, A.H.,
''How close are time series to power tail L\'evy diffusions? ``,  
{\it Chaos} {\bf 27}, (2017) 073112.
}

\bibitem{Gel}
Gelbrich, M.,  
''On a formula for the $L^2$ Wasserstein metric between measures on Euclidean and Hilbert Spaces``,  
{\it Math. Nachr.} {\bf 147}, (1990) 185-203.

{
\bibitem{GTR23}
Gilson, M., Tagliazucchi, E., and Cofr{\'e}, R.,
Entropy production of multivariate Ornstein-Uhlenbeck processes correlates with consciousness levels in the human brain
{\it Phys. Rev. E 107}, (2023) 024121.
}

\bibitem{Givens}
Givens, C.R., and Shortt, R.M.,  
''A class of Wasserstein metrics for probability distributions``, 
{\it Michigan Math. J.} {\bf 31}, no. 2, (1984) 231-240.

{
\bibitem{GL19}
Godr{\`e}che, C., and Luck, J.-M., 
Characterising the nonequilibrium stationary states
of Ornstein-Uhlenbeck processes
{\it J. Phys. A: Math. Theor.} 52 035002 (2019)
}

\bibitem{Ha15}
Hall, B., 
''Lie groups, Lie algebras, and representations. An elementary introduction.``,  Second edition, Springer Graduate Texts in Mathematics 222, 2015.

\bibitem{HS20}
Hermon, J., and Salez, J., 
''Cutoff for the mean-field zero-range process with bounded monotone rates``, 
{\it Ann. Probab.} {\bf 48}, no. 2 (2020) 742-759.



\bibitem{Ja17}
Jacobsen, M., 
''Laplace and the origin of the Ornstein-Uhlenbeck process``, 
\textit{Bernoulli} 2 (1996), no. 3, 271-286.

{
\bibitem{JPS17}
Jak\v{s}i\'{c}, V., Pillet,  C.,  and Shirikyan, A., 
''Entropic fluctuations in thermally driven harmonic networks``, 
{\it J. Stat. Phys.} {\bf 166}, 926-1015,
(2017).
}

\bibitem{Ja68}
Jameson, A., 
''Solution of the equation $AX+XB=C$ by inversion of an $M\times M$ or $N\times N$ matrix``, 
\textit{SIAM J. Appl. Math.} 16 (1968), 1020--1023.

{
\bibitem{JL14}
Janakiraman, D., Sebastian, K. L., 
''Unusual eigenvalue spectrum and relaxation in the L{\'e}vy-Ornstein-Uhlenbeck process'', 
{\it Phys Rev E Stat Nonlin Soft Matter Phys} (2014) 90(4):040101.
}


\bibitem{JohnssonTicozziViola17}
Johnson, P.D., Ticozzi, F., and  Viola, L., 
''Exact stabilization of entangled states in finite time by dissipative quantum circuits``, 
{\it Phys. Rev. A} {\bf 96}, 012308, (2017).

\bibitem{Jonson} 
Jonsson, G.F., and Trefethen, L.N.,  
''A numerical analysis looks at the ‘cut-off phenomenon’ in card shuffling
and other Markov chains.`` 
In: Numerical Analysis 1997 (Dundee 1997), pp. 150-178. Addison Wesley
Longman, Harlow (1998).

\bibitem{KS14}
Kallianpur, G., and Sundar, P.,  
''Stochastic analysis and diffusion processes``, Oxford Graduate
Texts in Mathematics 24. Oxford University Press, Oxford, 2014. xii+352 pp. MR-3156223

\bibitem{Kastoryano12} 
Kastoryano, M.J., Reeb, D., and Wolf, M.M.,  
A cutoff phenomenon for quantum Markov chains.
{\it J. Phys. A} {\bf 45}, 075307, (2012).

\bibitem{Kastoryano13} 
Kastoryano, M.J.,  Wolf, M.M., and Eisert, J.,  
''Precisely Timing Dissipative Quantum Information Processing``, 
{\it Phys. Rev. Lett.} {\bf 110}, 110501, (2013).

\bibitem{LL19} 
Labb{\'e}, C., and Lacoin, H.,  
''Cutoff phenomenon for the asymmetric simple exclusion process and the biased card shuffling``, 
{\it Ann. Probab.} {\bf 47}, no. 3 (2019) 1541-1586.


\bibitem{Lachaud2005}
Lachaud, B.,  
''Cut-off and hitting times of a sample of Ornstein-Uhlenbeck process and its
average``, {\it J. Appl. Probab.} {\bf 42}, no. 4 (2005) 1069-1080.

\bibitem{La16} 
Lacoin, H., 
''The cutoff profile for the simple exclusion process on the circle``,  
{\it Ann. Probab.} {\bf 44}, no. 5 (2016) 3399-3430.

\bibitem{LT85}
Lancaster, P., and Tismenetsky, M., 
''The Theory of Matrices``, 
{\it 2nd ed. Computer Science and
Applied Mathematics.} Academic Press, Orlando, FL., (1985) 

\bibitem{LanciaNardiScoppola12}
Lancia, C., Nardi, F., and Scoppola, B., 
''Entropy-driven cutoff phenomena``, 
{\it J. Stat. Phys.} {\bf 149}, no. 1 (2012) 108-141.

\bibitem{Langvin1908}
Langevin, P., 
''Sur la th{\'e}orie du mouvement brownien [On the Theory of Brownian Motion]``,  
{\it C. R. Acad. Sci. Paris.} 146: 530-533. (1908)

\bibitem{LLP10} 
Levin, D., Luczak, M., and Peres, Y., 
''Glauber dynamics for mean-field Ising model: cut-off, critical power law, and metastability``, 
{\it Probab. Theory Relat. Fields} {\bf 146}, no. 1 (2010) 223-265.

\bibitem{LPW} 
Levin, D, Peres, Y., and Wilmer, E., 
''Markov chains and mixing times``,  
{\it Amer. Math. Soc.}, Providence, 2009.

\bibitem{LW08}
Liang, T, and West, M. 
''Numerical Evidence for Cutoffs in Chaotic Microfluidic Mixing.`` 
in: Proceedings of the ASME 2008 Dynamic Systems and Control Conference. 
ASME 2008 Dynamic Systems and Control Conference, 
Parts A and B. Ann Arbor, Michigan, USA. October 20–22, 2008. pp. 1405--1412. 

\bibitem{LubetzkySly13}
Lubetzky, E., and Sly, A., 
''Cutoff for the Ising model on the lattice``, 
{\it Invent. Math.} {\bf 191}, no. 3 (2013) 719-755.

\bibitem{Mao2008} 
Mao X., 
''Stochastic differential equations and applications``,  
Second edition. Horwood Publishing Limited, Chichester, 2008.

\bibitem{MPZ19}
Masarotto, V., Panaretos, V. M., Zemel, Y., 
''Procrustes metrics on covariance operators and optimal transportation of Gaussian processes``, 
Sankhya A 81 (2019), no. 1, 172--213.

\bibitem{Ma04}
Masuda, H.,  
''On Multidimensional Ornstein-Uhlenbeck process driven by a general L\'evy process``, 
Bernoulli 10(1), 97--120 (2004)

\bibitem{Meliot14}
M\'eliot, P.-L., 
''The cut-off phenomenon for Brownian motions on compact symmetric spaces``, 
{\it Potential Anal.} {\bf 40}, no. 4 (2014) 427-509, .

\bibitem{Mikami} 
Mikami, T.,  
''Asymptotic expansions of the invariant density of a
Markov process with a small parameter``,  
\emph{Ann. Inst. H. Poincar\'e Sect. B} {\bf 24}, no. 3, (1988), 403--424.

\bibitem{Mi22}
Minh, H. Q., 
''Alpha Procrustes metrics between positive definite operators: a unifying formulation for the Bures-Wasserstein and log-Euclidean/log-Hilbert-Schmidt metrics``, 
\textit{Linear Algebra Appl.} 636 (2022), 25--68.

\bibitem{Pego17}
Murray, R.W., and Pego, R.L.,  
''Cutoff estimates for the Becker-D\"oring equations``,    
{\it Commun. Math. Sci.} {\bf 15},  1685-1702, (2017).

\bibitem{Pego16}
Murray, R.W., and Pego, R.L.,  
''Algebraic decay to equilibrium for the Becker-D\"oring equations``,   
{\it SIAM J. Math. Anal. } {\bf 48}, no. 4, 2819-2842, (2016).

\bibitem{OP82}
Olkin, I., Pukelsheim, F., 
''The distance between two random vectors with given dispersion matrices``, 
\textit{Linear Algebra Appl.} 48 (1982), 257--263.

\bibitem{dOTL18}
D'Onofrio, G., Tamborrino, M., Lansky, P., 
''The Jacobi diffusion process as a neuronal model`` 
\textit{Chaos} 28, 103119 (2018)

{
\bibitem{OPV14}
Optimal transportation.
Theory and applications. Including papers from the Summer School ``Optimal Transportation: Theory and Applications'' 
held at the University of Grenoble I. 
Ollivier, Y., Pajot H., Villani, C., (ed.), 
\textit{London Math. Soc. Lecture Note Ser.}, 413
Cambridge University Press, Cambridge, 2014. 
}

\bibitem{Panaretos}
Panaretos, V.M., and Zemel,  Y., 
''An invitation to statistics in Wasserstein space``, 
SpringerBriefs in Probability and Mathematical Statistics, 2020.

\bibitem{Pavliotis}
Pavliotis, G.A., 
''Stochastic Processes and Applications, Diffusion Processes, the Fokker-Planck and Langevin Equations``
Springer New York 2014. 

{
\bibitem{PC19}
Peyr{\'e}, G., and Cuturi, M., 
''Computational Optimal Transport: With Applications to Data Science``, 
\textit{Foundations and Trends in Machine Learning,} Vol. 11: No. 5-6,355-607, 
2019. 
}

\bibitem{PJDS14}
Pigoli, D., Aston, J. A. D., Dryden, I. L., Secchi, P., 
''Distances and inference for covariance operators``, 
\textit{Biometrika} 101 (2014), no. 2, 409--422.

\bibitem{Protter}
Protter, P., 
''Stochastic integration and differential equations. A new approach.`` 
Applications of Mathematics New York 21. Springer-Verlag, Berlin, 1990.

\bibitem{RAQUEPAS}
Raqu{\'e}pas, R.,     
''A note on Harris ergodic theorem, controllability and perturbations of harmonic networks``,  
{\it Ann. Henri Poincar{\'e}} {\bf 20},  605-629, (2019).


\bibitem{SP23+}
Santoro, L.V., and Panaretos, V.M.,  
''Large Sample Theory for Bures-Wasserstein Barycentres``, 
\textit{https://arxiv.org/abs/2305.15592}

{
\bibitem{SSB23}
Sarkar, R., Santra, I., and Basu, U., 
''Stationary states of activity-driven harmonic chains``,  
\textit{Phys. Rev. E} 107, (2023) 014123. 
}

\bibitem{Sato} 
Sato, K., 
''L{\'e}vy processes and infinitely divisible distributions``,  
Cambridge University Press (1999).

\bibitem{SaYa}
Sato, K., and Yamazato, M.,  
''Operator-self-decomposable distributions as limit distributions of processes of Ornstein-Uhlenbeck type.`` \emph{Stochastic Process. Appl.} {\bf 17}, no. 1, (1984), 73--100.

{
\bibitem{Sch13}
Schuss, Z., 
''Brownian dynamics at boundaries and interfaces: In physics, chemistry, and biology.``  
New York, Springer, 2013.
}

\bibitem{Simon11}
Simon, T.,  
''On the absolute continuity of multidimensional Ornstein-Uhlenbeck processes``, 
{\it Probab. Theory Related Fields} {\bf 151}, no. 1-2, (2011), 173-190.

{
\bibitem{SGA18}
Singh, R., Ghosh, D., and Adhikari, R., 
''Fast Bayesian inference of the multivariate Ornstein-Uhlenbeck process``,
{\it Phys. Rev. E } 98, (2018) 012136. 
}

\bibitem{vSmoluchovski1906}
von Smoluchowski, M., 
''Zur kinetischen Theorie der Brownschen Molekularbewegung und der Suspensionen``,  
{\it Ann. Phys. (Berlin)} 21 (14): 756-780, (1906)
https://doi.org/10.1002/andp.19063261405

\bibitem{Ta11}
Takatsu, A., 
''Wasserstein geometry of Gaussian measures``, 
Osaka J. Math. 48 (2011), no. 4, 1005--1026.

{
\bibitem{TL19}
Thomas, P.J., and Lindner, B., 
''Phase descriptions of a multidimensional Ornstein-Uhlenbeck process``,  
{\it Phys. Rev. E} 99, (2019) 062221. 
}

\bibitem{Trefethenbrothers}
Trefethen, L.N., and Trefethen, L.M., 
''How many shuffles to randomize a deck of cards? `` 
{\it Proc. R. Soc. Lond. Ser. A Math. Phys. Eng. Sci.}  
{\bf 456}, no. 8 (2000) 2561-2568.

\bibitem{OrensteinUhlenbeck1930}
Uhlenbeck, G. E., Ornstein, L. S., 
''On the theory of Brownian Motion``,  
{\it Phys. Rev.} 36 (5): 823-841. (1930)

\bibitem{Vernier20}
Vernier, E.,  
''Mixing times and cutoffs in open quadratic fermionic systems``  
\textit{SciPost Phys} \textbf{9}, no. 049, 1-30, (2020).


\bibitem{Vi09}
Villani, C. ,
''Optimal transport, old and new``, 
Springer, 2009.


\bibitem{WANGLETTERS} 
Wang, J.,  
''Exponential ergodicity and strong ergodicity for SDEs driven by symmetric $\alpha$-stable processes``,  \textit{Appl. Math. Lett.} {\bf 26}, no. 6, 654--658, (2013). 

\bibitem{WangJAP}
Wang, J.,  
''On the Exponential Ergodicity of L\'evy-Driven Ornstein-Uhlenbeck Processes``, 
\textit{Journal of Applied Probability}, 49(4), 990-1004 (2012). 

\bibitem{WC18}
Wang M., and Christov, I. C., 
''Cutting and shuffling with diffusion: Evidence for cut-offs in interval exchange maps`` 
Phys. Rev. E 98, 022221 (2018)


\bibitem{Yc99} 
Ycart, B.,  
''Cutoff for samples of Markov chains``,  
{\it ESAIM, Probab. Stat.} {\bf 3}  89-106 (1999). 

\bibitem{ZP19}
Zemel, Y., Panaretos, V. M.,
''Fr{\'e}chet means and Procrustes analysis in Wasserstein space``, 
Bernoulli 25 (2019), no. 2, 932--976.

\end{thebibliography}
\end{document}